\documentclass[11pt,leqno]{amsart}
\usepackage{epsfig}
\usepackage{amssymb}
\usepackage{amscd}
\usepackage[matrix,arrow]{xy}
\usepackage{graphicx}

\newtheorem{theo}{Theorem}[section]
\newtheorem{lemma}[theo]{Lemma}
\newtheorem{defi}[theo]{Definition}
\newtheorem{prop}[theo]{Proposition}

\newtheorem{cor}[theo]{Corollary}
\newtheorem{remark}[theo]{Remark}

\numberwithin{equation}{section}

\def\R{\mathbb{R}}
\def\C{\mathbb{C}}
\def\Z{\mathbb{Z}}

\def\F{{\mathcal F}}

\def\bR{{\mathbf R}}

\def\PP{{\mathbb P}}

\def\pre-tr{\operatorname{pre-tr}}

\def\Hom{\operatorname{Hom}}
\def\End{\operatorname{End}}

\newcommand{\mk}{\mathrm k}

\newcommand{\cF}{{\mathcal F}}

\newcommand{\cO}{{\mathcal O}}

\newcommand{\cL}{{\mathcal L}}
\newcommand{\cM}{{\mathcal M}}

\newcommand{\cA}{{\mathcal A}}
\newcommand{\cB}{{\mathcal B}}

\newcommand{\cC}{{\mathcal C}}

\newcommand{\cU}{{\mathcal U}}

\newcommand{\mg}{\mathfrak g}
\newcommand{\mh}{\mathfrak h}

\newcommand{\diag}{\operatorname{diag}}
\newcommand{\sgn}{\operatorname{sgn}}

\newcommand{\Perf}{\operatorname{Perf}}

\newcommand{\supp}{\operatorname{Supp}}

\newcommand{\coker}{\operatorname{Coker}}

\newcommand{\Ext}{\operatorname{Ext}}

\newcommand{\Id}{\operatorname{Id}}

\newcommand{\ord}{\operatorname{ord}}

\newcommand{\Sym}{\operatorname{Sym}}

\newcommand{\Hilb}{\rm Hilb}

\newcommand{\Ho}{\operatorname{Ho}}

\newcommand{\id}{\operatorname{id}}

\newcommand{\ev}{\operatorname{ev}}

\usepackage{epsf}
\usepackage{amscd}

\title[Homological mirror symmetry for curves of higher genus]
{Homological mirror symmetry for curves of higher genus}

\author{Alexander I. Efimov}
\address{Steklov Mathematical Institute of RAS, Gubkin str. 8, GSP-1, Moscow 119991, Russia} \email{efimov@mccme.ru}
\thanks{The author was partially supported by
the Moebius Contest Foundation for Young Scientists, RFBR (grant 4713.2010.1) and "Dynasty" foundation.}

\begin{document}

\begin{abstract} This paper is devoted to homological mirror symmetry conjecture for curves of higher genus.
It was proposed by Katzarkov as a generalization of original Kontsevich's conjecture.

A version of this conjecture in the
case of the genus two curve was proved by Seidel \cite{Se}. Based on the paper of Seidel, we prove the
conjecture (in the same version) for curves of genus $g\geq 3.$
Namely, we relate the Fukaya category of a genus $g$ curve to the category
of singularities of zero fiber in the mirror dual Landau-Ginzburg model.

We also prove a kind of reconstruction theorem for hypersurface
singularities. Namely, formal type of hypersurface singularity
(i.e. a formal power series up to a formal change of variables)
can be reconstructed, with some technical assumptions, from its
D$(\Z/2)\text{-}$G category of Landau-Ginzburg branes. The precise
statement is Theorem \ref{reconstr_intro}.
\end{abstract}

\maketitle


\section{Introduction}

The homological mirror symmetry conjecture is a categorical interpretation of mirror symmetry. Originally, it
was proposed by Kontsevich \cite{Kon} for
Calabi-Yau varieties. It was proved in some special cases
\cite{AS, PZ, Se3}.

An analogue of the conjecture for
Fano varieties has been proposed soon after. In this case the mirror is a Landau-Ginzburg model --- a smooth algebraic variety
together with a regular function. More generally, it is believed that one can consider varieties with effective anti-canonical divisor,
 see \cite{Au}.

Katzarkov \cite{Ka, KKP, KKOY} has proposed a generalization of
Homological Mirror Symmetry, which includes some varieties of
general type. The mirror to such variety is a Landau-Ginzburg model.
One
direction of Katzarkov's conjecture was proved by Seidel in the
case of the genus $2$ curve \cite{Se}. The main aim of this paper is to
prove it in the case of curves of genus $g\geq 3.$
Actually, we follow the steps of Seidel's proof in the genus $2$ case, and generalize it to genus $g\geq 3$ case.

We treat genus $\geq 3$ curves as symplectic varieties, and associate to them Fukaya categories. Further, Landau-Ginzburg models
 are considered algebro-geometrically. The associated categories are
 triangulated categories of singularities of singular fiber \cite{Or}.

Let $M$ be a symplectic compact oriented surface of genus $g\geq 3.$
The mirror Landau-Ginzburg (LG for short)
model $W:X\to \C$ is three-dimensional. The only singular fibre $H:=X_0\subset X$ is a union of
$(g+1)$ surfaces. This LG model will
be constructed explicitly in Section \ref{LG_model}.

We denote by $\cF(M)$ the Fukaya
$A_{\infty}$-category of M, and by $D^{\pi}(\cF(M))$ the category of perfect complexes over $\cF(M).$ Further, let $D_{sg}(H)$
be the category of singularities of the surface $H$, and denote by
$\overline{D_{sg}}(H)$ its Karoubian completion. The main result of the paper is the following.

\begin{theo} The triangulated categories
$D^{\pi}(F(M))$ and $\overline{D_{sg}}(H)$ are equivalent.\end{theo}

The main ideas in the proof are the same as in \cite{Se}. We sketch the steps of the proof.

Take $V = \C^3.$ We denote by $\xi_k\in V,$ $k=1,2,3$ the standard basis vectors
of $V,$ and by $z_k\in V^*,$ $k=1,2,3$  the
dual basis. Take the $K$-invariant polynomial
\begin{equation}\label{potential}W = -z_1z_2z_3 + z_1^{2g+1} + z_2^{2g+1} + z_3^{2g+1}\in \C[V^{\vee}]^{K},\end{equation}
where $K\cong \Z/(2g+1)\subset SL(V)$ is the cyclic subgroup generated by
the diagonal matrix $\diag(\zeta, \zeta, \zeta^{2g-1}),$ with
$\zeta=\exp(\frac{2\pi i}{2g+1}).$

{\it A generator of Fukaya category.} The generator of $D^{\pi}(\cF(M))$ is constructed as follows. We consider a cyclic covering
$\pi:M\to \bar{M},$ where $\bar{M}$ is  $\PP^1$ with three orbifold points.
The Galois group of this covering is
$\Sigma=\Hom(K,\C^*)\cong \Z/(2g+1).$ There is a nice Galois-invariant collection of
curves $L_1,\dots,L_{2g+1}\subset M,$ such that

1) the object $L_1\oplus\dots\oplus L_{2g+1}\in D^{\pi}(\cF(M))$ is a generator;

2) the projection $\pi(L_i)$ of each of these curves is the immersed curve
$\bar{L}\subset \bar{M}.$

Here to prove generation we use the criterions of Seidel (\cite{Se}, Lemmas 6.4 and 6.5). The endomorphism $A_{\infty}$-algebra
$\End(\bigoplus_{1\leq i\leq 2g+1} L_i)$ is a smash product $\End(\bar{L})\#\C[K].$  The Floer cohomology super-algebra $HF^{\cdot}(\bar{L},\bar{L})$ is
isomorphic to the exterior super-algebra $\Lambda(V).$ We compute some
higher $A_{\infty}$-operations which uniquely determine the whole $A_{\infty}$-structure (up to homotopy). This computation is analogous to that of \cite{Se}, Section 10, and is in fact combinatorial, as in the approach of Abouzaid \cite{Ab}.

{\it Classification of $A_{\infty}$-structures.} The super-algebra $\Lambda(V)$ has a lot of (homotopy classes of) $\Z/2\text{-}$graded
$A_{\infty}\text{-}$structures. These $A_{\infty}$-structures are actually Maurer-Cartan solutions in the differential graded Lie algebra of Hochschild
cochains. We use Kontsevich's formality theorem
\cite{Ko} (in the suitable version) to reduce classification of $A_{\infty}$-structures
to some questions on formal polyvector fields on $V.$ It turns out
that the $A_{\infty}\text{-}$algebra $\End(\bar{L})$ above (which gives an $A_{\infty}$-structure on $\Lambda(V)\cong HF^{\cdot}(\bar{L},\bar{L})$), corresponds to the (gauge equivalence
class of) the superpotential $W$ (considered as a polyvector field). This part of the paper generalizes \cite{Se}, sections 4 and 5. Technical details here are more complicated than in \cite{Se}.

{\it Matrix factorizations.} It is well known that the triangulated category of singularities of a fiber $W^{-1}(0)$ is equivalent to the homotopy category of matrix
factorizations of $W$ \cite{Or}. In our case, the structure sheaf
of the origin $\cO_0$ is a split-generator in the category of singularities. We take
the matrix factorization corresponding to this skyscraper sheaf
$\cO_0$. The endomorphism DGA of this matrix factorization turns
out to be quasi-isomorphic to the $A_{\infty}\text{-}$algebra
computed on the Fukaya side. Namely, the cohomology super-algebra
of this DGA is isomorphic to the exterior algebra $\Lambda(V)$ and
again the resulted $A_{\infty}\text{-}$structure corresponds to
the superpotential $W$ in polyvector fields. This part generalizes \cite{Se}, sections 11, 12.

Here we also prove
the following general reconstruction theorem (more precise formulation is Theorem
\ref{reconstr}):

\begin{theo}\label{reconstr_intro} Let $\mk$ be a field of characteristic zero, $n\geq 1,$ and $V=\mk^n.$ Let $W=\sum\limits_{i=3}^dW_i\in
\mk[V^{\vee}]$be a non-zero polynomial, where $W_i\in
\Sym^i(V^{\vee}).$ Then $W$ can be reconstructed, up to a formal
change of variables, from the quasi-isomorphism class of
D$(\Z/2)\text{-}$G algebra $\cB_W\cong
\bR\Hom_{D_{sg}(W^{-1}(0))}(\cO_0,\cO_0),$ the endomorphism D$(\Z/2)\text{-}$G
algebra of the structure sheaf $\cO_0$ in
$D_{sg}(W^{-1}(0)),$ together with identification
$H^{\cdot}(\cB_W)\cong \Lambda(V).$ Moreover, formal change of
variables is of the form
\begin{equation}z_i\to z_i+O(z^2).\end{equation}
\end{theo}

{\it Equivalence between two LG models.} We have two natural LG models both mirror to the curve $M.$ The first one
is a stack $V//K$ together with a function $W$ from equation \eqref{potential}. Another one is a crepant resolution $\psi:X\to \bar{X} = V/K$ given by the $K\text{-}$Hilbert scheme
\cite{CR}, together with pullback of $W.$ In both cases the only singular fiber is over zero. Denote by $H\subset X$ be the preimage of $\bar{H} = W^{-1}(0)/K
\subset \bar{X}.$ We can describe the surface $H$ very explicitly (Section \ref{LG_model}). By the famous Mckay correspondence for derived categories \cite{BKR}, we have an equivalence $D^b_K (V )
\cong D^b(X).$ We use an analogous result for categories of singularities
\cite{BP, QV}: $D_{sg,K}(W^{-1}(0))\cong D_{sg}(H).$
This is a generalization of \cite{Se}, section 13.

In Appendix we prove one necessary technical result from Maurer-Cartan theory for pro-nilpotent DG Lie algebras.

\noindent{\bf The sign convention.} We will treat an
$A_{\infty}\text{-}$algebra as a $\Z\text{-}$(or
$(\Z/2)\text{-}$)graded vector space equipped with a sequence of
maps $\mu^d:A^{\otimes d}\to A$ of degree $2-d$ (resp. of parity
$d$) such that the maps $m_d:A^{\otimes d}\to A,$ where
\begin{equation}m_d(a_d,\dots,a_1)=(-1)^{|a_1|+2|a_2|+\dots+d|a_d|}\mu^d(a_d,\dots,a_1),\end{equation}
define an $A_{\infty}\text{-}$structure in standard sign convention.

\noindent{\bf Acknowledgements.} I am grateful to D. Kaledin, A.
Kuznetsov, S. Nemirovski, D. Orlov and P. Seidel for their help
and useful discussions.

\section{Maurer-Cartan theory for pro-nilpotent DG Lie algebras}
\label{MC_general}

Let $\mg$ be some DG Lie algebra over $\C.$ Recall Maurer-Cartan (MC) equation for $\mg:$
\begin{equation}\label{MC}\partial \alpha+\frac12 [\alpha,\alpha]=0,\quad \alpha\in\mg^1.\end{equation}
An element $\alpha\in
\mg^1$ is called Maurer-Cartan (MC) element if it satisfies
MC equation. For each $\gamma\in\mg^0$ we have affine vector field on $\mg^1,$ $\alpha\mapsto -\partial \gamma+ [\gamma,\alpha].$ This defines a morphism of Lie algebras from $\mg^0$ to the Lie
algebra of affine vector fields on $\mg^1.$ It
is easy to check that all vector fields in the image are tangent
to the subscheme of solutions of \eqref{MC}. Under some natural assumptions on $\mg$ (see below),
there is a group $G^0$ (which is exponent of $\mg^0$) acting on the set
of Maurer-Cartan elements.

We will need to deal with $L_{\infty}\text{-}$morphisms between DG
Lie algebras. An $L_{\infty}$-morphisn $\Phi:\mg\to \mh$ is given by a
sequence of maps $\Phi^k:\mg^{\otimes k}\to \mh.$ These maps must be anti-symmetric (in super sense) and satisfy natural compatibility equations (\cite{LM}, Definition 5.2).

More precisely, for a permutation $\sigma\in S_n,$ and graded variables $x_1,\dots,x_n,$ define the {\it Koszul sign} by the equality
$$x_1\wedge\dots\wedge x_n=\epsilon(\sigma;x_1,\dots,x_n)\cdot x_{\sigma(1)}\wedge\dots\wedge x_{\sigma(n)}$$
in the free graded commutative algebra $\Lambda(x_1,\dots,x_n).$ Further,
put $\chi(\sigma)=\chi(\sigma;x_1,\dots,x_n):=\sgn(\sigma)\cdot\epsilon(\sigma;x_1,\dots,x_n).$ Then the maps $\Phi^k$ must satisfy the equations $$\Phi^k(\xi_1,\dots,\xi_k)=\chi(\sigma)\Phi^k(\xi_{\sigma(1)},\dots,\xi_{\sigma(k)})$$
for homogeneous $\xi_1,\dots,\xi_k\in\mg,$ $k\geq 1.$
Further, the following relations are required to hold:
\begin{multline*}\partial\Phi^n(\xi_1,\dots,\xi_n)+\frac{(-1)^n}{(n-1)!}\sum\limits_{\sigma\in S_n}\chi(\sigma)
\Phi^n(\partial \xi_{\sigma(1)},\xi_{\sigma(2)},\dots,\xi_{\sigma(n)})\\
-\frac{1}{2!(n-2)!}\sum\limits_{\sigma\in S_n}\chi(\sigma)
\Phi^n([\xi_{\sigma(1)},\xi_{\sigma(2)}],\xi_{\sigma(3)},\dots,\xi_{\sigma(n)})\\
+\sum\limits_{s+t=n}\frac{1}{s!l!}\sum\limits_{\tau\in S_n}\chi(\tau)(-1)^{s-1}(-1)^{(t-1)(\sum\limits_{p=1}^s)|\xi_{\tau(p)}|}
[\Phi^s(\xi_{\tau(1)},\dots,\xi_{\tau(s)}),\Phi^t(\xi_{\tau(s+1)},\dots,\xi_{\tau(n)})]\\=0,
\end{multline*}
where again $\xi_1,\dots,\xi_n$ are homogeneous elements of $\mg,$ $n\geq 1.$

In particular, $\Phi^1$ is a morphism of
complexes, and $H^{\cdot}(\Phi^1):H^{\cdot}(\mg)\to H^{\cdot}(\mh)$ is a morphism of graded Lie algebras.

Such $\Phi$ is called a quasi-isomorphism if $\Phi^1$ is a
quasi-isomorphism. We will need the following statement.

\begin{lemma}\label{quasi-inverse} Let $\mg$ be a graded Lie algebra considered as a DG Lie algebra
with zero differential. Let $\mh$ be a DG Lie algebra, and
$\Psi:\mg\to \mh$ an $L_{\infty}\text{-}$quasi-isomorphism. Take
some morphism of complexes $\Phi^1:\mh\to \mg$ together with a
homogeneous map $H:\mh\to \mh$ of degree $-1,$ such that

\begin{equation}\Phi^1\Psi^1=\id,\quad \Psi^1\Phi^1-\id=\partial H+H\partial.\end{equation}
Then $\Phi^1$ can be extended to an $L_{\infty}\text{-}$morphism
$\Phi:\mh\to \mg,$ so that the higher order terms $\Phi^k$ are
given by a universal formulae, depending only on $\Psi,$ $\Phi^1$
and $H.$

Moreover, one can choose $\Phi$ in such a way that the composition
$\Phi\circ \Psi$ equals to the identity
$L_{\infty}\text{-}$morphism.
\end{lemma}
\begin{proof} For the proof of the first statement, see \cite{Se}, Lemma 3.1. Further, for the
constructed $\Phi,$ we have that the composition $\Phi\circ\Psi$
is an $L_{\infty}\text{-}$automorphism of $\mh.$ Define
$\Phi'=(\Phi\circ\Psi)^{-1}\Phi.$ Then $\Phi'$ satisfies the
required property, and the higher order terms $\Phi'^k$ are again
given by a universal formulae, depending only on $\Psi,$ $\Phi^1$
and $H.$
\end{proof}

In order to be able to exponentiate the gauge vector fields on
$\mg^1,$ we will deal with {\it pro-nilpotent} DG Lie algebras.

\begin{defi} A DG Lie algebra $\mg$ is called pro-nilpotent if it
is equipped with a complete decreasing filtration
$\mg=L_1\mg\supset L_2\mg\supset\dots,$ such that
\begin{equation}\partial(L_r\mg)\subset L_r\mg,\quad [L_r\mg,L_s\mg]\subset L_{r+s}\mg.\end{equation}
\end{defi}

If $\mg$ is pro-nilpotent, then Lie algebra $\mg^0$ is also such,
and hence we get a pro-nilpotent group $G^0.$ As a set, it equals to $\mg^0,$ and the product is given by the
Baker-Campbell-Hausdorff formula. The group $G^0$ then acts on MC
elements $\alpha\in \mg^1.$ Two MC elements are called equivalent if they lie in the same $G^0$-orbit.

\begin{defi} Let $\mg,$ $\mh$ be pro-nilpotent DG Lie algebras. An $L_{\infty}\text{-}$morphism
$\Phi:\mg\to \mh$ is called filtered if
\begin{equation}\Phi^k(L_{r_1}\mg\otimes\dots\otimes L_{r_k}\mg)\subset L_{r_1+\dots+r_k}\mh.\end{equation}
\end{defi}

\begin{defi} A filtered $L_{\infty}\text{-}$morphism $\Phi:\mg\to\mh$ of pro-nilpotent DG Lie algebras
is called a filtered $L_{\infty}\text{-}$quasi-isomorphism if the
induced morphisms of complexes $L_r\mg/L_{r+1}\mg\to
L_r\mh/L_{r+1}\mh$ are quasi-isomorphisms.\end{defi}

\begin{remark}In Lemma \ref{quasi-inverse} we can require $\mg,$ $\mh$ to be pro-nilpotent, $\Psi$ to be a filtered $L_{\infty}\text{-}$quasi-isomorphism,
and $\Phi^1,$ $H$ to be compatible with filtrations. Then the
constructed $L_{\infty}\text{-}$morphism $\Phi$ is also
filtered.\end{remark}

If $\Phi:\mg\to \mh$ is a filtered $L_{\infty}\text{-}$morphism of
pro-nilpotent DG Lie algebras, then we have an induced map on
Maurer-Cartan elements

\begin{equation}\alpha\mapsto \Phi_*(\alpha):=\sum\limits_{k\geq 1} (-1)^{\frac{k(k-1)}{2}}\frac1{k!}\Phi^k(\alpha,\dots,\alpha).\end{equation}
This map preserves equivalence relation (see Appendix). The
following statement is a generalization of the corresponding
result in \cite{Ko}.

\begin{lemma}\label{bijection} Let $\Phi:\mg\to \mh$ be a filtered
$L_{\infty}\text{-}$quasi-isomorphism of filtered DG Lie algebras.
Then the induced map on equivalence classes of MC elements is a
bijection.
\end{lemma}

This lemma is proved in Appendix by using obstruction theory,
similar to \cite{GM} (or \cite{ELO} for
$A_{\infty}\text{-}$algebras).

\section{$A_{\infty}$-structures and formal polyvector fields}
\label{intro_to_formality}

Now we define some necessary notions to formulate a version of Kontsevich formality theorem
\cite{Ko}. Let $V$ be a
finite-dimensional $\C\text{-}$vector space. The graded
Lie algebra of formal polyvector fields on $V$ is the following:
\begin{equation}\C[[V^{\vee}]]\otimes \Lambda(V)=\prod\limits_{i,j}\Sym^i(V^{\vee})\otimes \Lambda^j(V).\end{equation}

We assign to the summand $\C[[V^{\vee}]]\otimes \Lambda^j(V)$
the grading $j-1.$ The Lie bracket is the Schouten one:

\begin{multline}[f\xi_{i_1}\wedge\dots\wedge\xi_{i_k},g\xi_{j_1}\wedge\dots\wedge\xi_{j_l}]=\\
\sum\limits_{q=1}^k(-1)^{k-q}(f\partial_{i_q}g)\xi_{i_1}\wedge\dots\wedge\widehat{\xi_{i_q}}\wedge\dots\wedge\xi_{i_k}
\wedge\xi_{j_1}\wedge\dots\wedge\xi_{j_l}+\\
\sum\limits_{p=1}^l(-1)^{l-p-1+(k-1)(l-1)}(g\partial_{j_p}f)\xi_{j_1}\wedge\dots\wedge
\widehat{\xi_{j_p}}\wedge\dots\wedge\xi_{j_l}
\wedge\xi_{i_1}\wedge\dots\wedge\xi_{i_k}.
\end{multline}

A formal bivector field $\alpha\in \C[[V^{\vee}]]\otimes \Lambda^2(V)$ is MC element
iff $\alpha$ defines a formal Poisson structure. The
elements $\gamma\in \C[[V^{\vee}]]\otimes V,$ which are formal
vector fields, act on Poisson brackets by their Lie derivatives.
If the value of $\gamma$ at the origin vanishes, then it can be
exponentiated to a formal diffeomorphism of $V.$ The corresponding
action on Poisson brackets is just the pushforward action by
formal diffeomorphisms.

Now let $A$ be a graded algebra over $\C.$ The Hochshild cochain
complex $CC^{\cdot}(A,A)$ of $A$ is defined as follows. As a graded vector space, it consists of graded multilinear maps:
\begin{equation}CC^d(A,A)=\prod\limits_{i+j-1=d}\Hom^j(A^{\otimes
i},A).\end{equation} The differential on the Hochshild complex is given by the formula
\begin{multline} (\partial\phi)^j (a_j,\dots,a_1) =
\sum\limits_{k}(-1)^{|\phi|+|a_1|+\dots+|a_k|+k}\phi^{j-1}(a_j,\dots,
a_{k+1}a_k,\dots,a_1)+\\
(-1)^{|\phi|+|a_1|+\dots+|a_{j-1}|+j}a_j\phi^{j-1}(a_{j-1},
\dots,a_1)+\\(-1)^{(|\phi|-1)(|a_1|-1)+1}\phi^{j-1}(a_j,\dots,a_2)a_1.
\end{multline}

There is a naturaL Gerstenhaber bracket on the Hochshild complex which makes it into a DG Lie algebra:
\begin{multline}[\phi,\psi]^j(a_j,\dots,a_1)=\\\sum\limits_{k,l}(-1)^{|\psi|(|a_1|+\dots+|a_k|-k)}
\phi^{j-l+1}(a_j,\dots,a_{k+l+1},\psi^l(a_{k+l},\dots,a_{k+1}),a_k,\dots,a_1)-\\
\sum\limits_{k,l}(-1)^{|\phi||\psi|+|\phi|(|a_1|+\dots+|a_k|-k)}
\psi^{j-l+1}(a_j,\dots,a_{k+l+1},\phi^l(a_{k+l},\dots,a_{k+1}),a_k,\dots,a_1).
\end{multline}

Our grading on the Hochshild complex is shifted by $1$ from the usual one (otherwise the Gerstenhaber bracket would have degree $-1$).

We would like to illustrate the Maurer-Cartan theory for pro-nilpotent DG Lie algebras by describing minimal $A_{\infty}$-structures on $A$ up to a strict homotopy.
Consider the DG Lie subalgebra $\mg_A\subset CC^{\cdot}(A,A)$ with
\begin{equation}\mg_A^d=\prod\limits_{\substack{i+j-1=d,\\i\geq d+2}}\Hom^j(A^{\otimes
i},A).\end{equation}
We have that $\mg_A$ is pro-nilpotent, with filtration
\begin{equation}L_r\mg_A^d=\prod\limits_{\substack{i+j-1=d,\\i\geq d+1+r}}\Hom^j(A^{\otimes
i},A),\quad r\geq 1.\end{equation}

It is well known (and is easy to see) that $A_{\infty}$-structures on the graded algebra $A$ correspond to MC elements
$\alpha\in CC^1(A,A).$ Namely, each $\alpha\in CC^1(A,A)$ is given by maps
$\alpha^j:A^{\otimes j}\to A$ of degree $2-j,$ for each $j\geq 3.$ Put
\begin{equation} \begin{cases}& \mu^j=\alpha^j  \text{ for }j\geq 3;\\
& \mu^2(a_2,a_1)=(-1)^{|a_1|}a_2a_1;\\
& \mu_1=0.
\end{cases}
\end{equation}
Then $\mu^j$ define an $A_{\infty}\text{-}$structure if and only if $\alpha$ is Maurer-Cartan element.

\begin{remark} As we have already mentioned in Introduction, our sign convention differs from the standard one.
To obtain an $A_{\infty}\text{-}$structure in standard sign
convention, one should put
\begin{equation}m_j(a_j,\dots,a_1)=(-1)^{|a_1|+2|a_2|+\dots+j|a_j|}\mu^j(a_j,\dots,a_1).\end{equation}
\end{remark}

The exponentiated action of $\exp(\mg_A^0)$ on MC elements ($A_{\infty}$-structures) is the following.
Take some $\gamma\in \mg_A^0.$ Take homogeneous maps $\phi^r:A^{\otimes r}\to A,$ $\deg(\phi^r)=1-r,$ $r\geq 1,$ given by the formulas:
\begin{equation}\label{exp_gamma}\begin{cases}&
\phi^1=\id;\\
&
\phi^2=\gamma^2;\\
&
\phi^3=\gamma^3+\frac12 \gamma^2(\gamma^2\otimes\id)+\frac12 \gamma^2(\id\otimes\gamma^2);\\
&
\phi^4=\gamma^4+\frac12 \gamma^2(\gamma^3\otimes\id)+\frac12 \gamma^2(\id\otimes\gamma^3)+
\frac12 \gamma^3(\gamma^2\otimes\id\otimes\id)+\frac12 \gamma^3(\id\otimes\gamma^2\otimes\id)+\\
& \frac12 \gamma^3(\id\otimes\id\otimes\gamma^2)+
\frac13 \gamma^2(\gamma^2\otimes\gamma^2);\\
& \dots\end{cases}
\end{equation}

In general, $\phi^j$ is the sum over all ways of
concatenating the components of $\gamma$ to get a
$j\text{-}$linear map. The associated term is taken with the
coefficient $\frac s{r!},$ where $r$ is the number of components of $\gamma,$ and $s$ is the number of ways of
ordering the components, compatibly with their appearance in
concatenation. If two MC elements $\alpha$ and $\tilde{\alpha}$ lie in the same
orbit of the action of $\mg_A^0$, so that $\tilde{\alpha}=\exp(\gamma)(\alpha),$ then
the
corresponding $A_{\infty}\text{-}$structures are strictly homotopic,
and $\phi$ is an $A_{\infty}\text{-}$isomorphism.

Now let again $V$ be a finite-dimensional vector space, and take
$A=\Lambda(V).$ By Hochshild-Kostant-Rosenberg Theorem (see \cite{HKR}), we have
$HH^{\cdot}(A,A)\cong \C[[V^{\vee}]]\otimes \Lambda(V).$ This
isomorphism is induced by Hochshild-Kostant-Rosenberg map
\begin{equation}\Phi^1:CC^{\cdot}(A,A)\to \C[[V^{\vee}]]\otimes \Lambda(V),\end{equation}
given by the formula
\begin{equation} \Phi^1(\beta)(\xi)=\sum\limits_{j\geq
1}\beta^j(\xi,\dots\xi).
\end{equation}
Here we consider polyvector fields as formal power series with values in
$\Lambda(V).$

\begin{theo}\label{formality}(\cite{Ko}) The map $\Phi^1$ is the first term of some
$L_{\infty}\text{-}$morphism $\Phi,$ which can be taken to be
$GL(V)\text{-}$equivariant.
\end{theo}

Theorem \ref{formality} is implied by Kontsevich formality Theorem \cite{Ko}
using Lemma \ref{quasi-inverse} and reductiveness of $GL(V),$ see
\cite{Se} and Remark \ref{left_inverse}.

\begin{remark}\label{left_inverse}In contrast to our situation, Kontsevich deals with the algebra of smooth functions
on smooth manifolds. He proves that for each smooth manifold $X$
the graded Lie algebra of polyvector fields $T_{poly}(X)$ is
quasi-isomorphic to the DG Lie algebra of polydifferential
operators $D_{poly}(X)$. In the case when $X$ is an open domain
$U$ in affine space $\R^d,$ he constructs an explicit
$L_{\infty}\text{-}$quasi-isomorphism. One can replace the smooth
functions by polynomials (or formal power series) over $\C,$ and
his construction works as well. Then one exchanges even an odd
variables, and obtains an $L_{\infty}\text{-}$quasi-isomorphism
\begin{equation}\Psi:\C[[V^{\vee}]]\otimes \Lambda(V)\to
CC^{\cdot}(A,A).
\end{equation}
This $\Psi$ is $GL(V)\text{-}$equivariant, and using Lemma
\ref{quasi-inverse} and reductiveness of $GL(V),$ one obtains the
required $\Phi:CC^{\cdot}(A,A)\to \C[[V^{\vee}]]\otimes
\Lambda(V)$ which can be taken to be left inverse to
$\Psi.$\end{remark}

\section{Classification lemma for polyvector fields}

Put $V=\C^3.$ Take the subgroup $G\subset SL(V)$ which consists of
diagonal matrices with $(2g+1)\text{-}$th roots of unity on the
diagonal. Clearly, $G\cong (\Z/(2g+1))^2.$ Define the
pro-nilpotent graded Lie algebra $\mg$ as follows:

\begin{equation}\mg^d \;\; = \!\!\!\!\! \prod_{\substack{2i+j-(4g-4)k = 3d+3 \\
k \geq 0, \, i \geq d+2}}\!\!\!\!\! (\Sym^iV^\vee \otimes
\Lambda^jV)^G\, \hbar^k.\end{equation}

The Lie bracket comes from Schouten bracket on polyvector fields,
and $L_r \mg^d$ is the part of the product which consists of terms
with $i\geq d+1+r.$

We can omit $\hbar^k$ but remember that
\begin{equation}\label{conditions} 2i+j-3d-3\geq 0,\quad \text{and}\quad
2i+j-3d-3\equiv 0\text{ mod }4g-4.\end{equation}

We would like to describe explicitly elements of $\mg^1$ and $\mg^0,$ and Maurer-Cartan equation.
Any element $\alpha\in \mg^1$ can be written as
$(\alpha^0,\alpha^2),$ where $\alpha^0\in \C[[V^{\vee}]],$ and
$\alpha^2\in \C[[V^{\vee}]]\otimes \Lambda^2 V.$ Both $\alpha^0$ and $\alpha^2$ must be $G$-invariant,
and the degrees of non-zero homogeneous components of $\alpha^0$ and $\alpha^2$ must fullfill the conditions \eqref{conditions}. In particular,
$\alpha^0\in F_3\C[[V^{\vee}]],$ and $\alpha^2\in F_{2g}\C[[V^{\vee}]]\otimes\Lambda^2 V.$
 Here $F_{\bullet}\C[[V^{\vee}]]$ is the complete decreasing
filtration, s.t.
\begin{equation}F_r\C[[V^{\vee}]]=\prod_{i\geq r}\Sym^i(V^{\vee}).\end{equation}

Similarly, any element
$\gamma\in \mg^0$ can be written as $(\gamma^1,\gamma^3),$ where
$\gamma^1\in F_{2g-1}\C[[V^{\vee}]]\otimes V,$ and $\gamma^3\in
F_{2g-2}\C[[V^{\vee}]]\otimes \Lambda^3 V.$ Again, both $\gamma^1$ and $\gamma^3$ must be $G$-invariant, and non-zero homogeneous components of $\gamma^1$
and $\gamma^3$ must satisfy \eqref{conditions}.

Maurer-Cartan equation for $\alpha=(\alpha^0,\alpha^2)$ splits
into the components: \begin{equation}\frac12
[\alpha^2,\alpha^2]=0,\quad [\alpha^0,\alpha^2]=0.\end{equation}

This means that

1) The bivector field $\alpha^2$ is Poisson (the first equation);

2) The Poisson vector
field associated to the function $\alpha^0$ is identically zero.
It will be convenient to reformulate this. Consider the complex $\C[[V^{\vee}]]\otimes \Lambda^{\cdot} (V)$ with differential being
contraction with $d\alpha^0$ (Koszul complex). Then the second equation means that
$\alpha^2$ is a cocycle in this complex.

The exponentiated adjoint action of $\gamma=(\gamma^1,0)\in \mg^0$
on the solutions of MC equation is the usual action by formal
diffeomorphisms. For $\gamma=(0,\gamma^3),$ this action is given
by the formula
\begin{equation}\label{act_3_forms}(\alpha^0,\alpha^2)\mapsto
(\alpha^0,\alpha^2+\iota_{d\alpha^0}\gamma^3).\end{equation}

Take the polynomial
\begin{equation}W=-z_1z_2z_3+z_1^{2g+1}+z_2^{2g+1}+z_3^{2g+1}\in
\C[V^{\vee}]^G,\end{equation}
which we have already mentioned in Introduction as a superpotential. Then $(W,0)\in \mg^1$ is a solution
of MC equation (as any other $\alpha\in\mg^1$ of type
$(\alpha^0,0)$). Our main technical result in this section is the following.

\begin{lemma}\label{equiv_to_(W,0)} Let $\alpha=(\alpha^0,\alpha^2)\in \mg^1$ be an MC element. Suppose that
\begin{equation}\alpha^0\equiv\begin{cases}W\text{ mod }F_{2g+2}\C[[V^{\vee}]] &
\text{ if }g\not\equiv 1\text{ mod }3\\
W+\lambda (z_1z_2z_3)^{\frac{2g+1}3},\text{ where }\lambda\in \C &
\text{ if }g\equiv 1\text{ mod }3.\end{cases}\end{equation}
Then $\alpha$ is
equivalent to $(W,0).$
\end{lemma}
\begin{proof}First we note that in the case ($g\equiv 1\text{ mod
}3$) one may assume that $\lambda=0.$ Indeed, in this case we have
\begin{multline}\exp(\lambda z_1^{\frac{2g+1}3}z_2^{\frac{2g-2}3}z_3^{\frac{2g-2}3}\otimes \xi_1)^*\alpha^0\equiv\\
\alpha^0+\lambda
z_1^{\frac{2g+1}3}z_2^{\frac{2g-2}3}z_3^{\frac{2g-2}3}\frac{\partial
\alpha^0}{\partial z_1}\equiv W\text{ mod
}F_{2g+2}\C[[V^{\vee}]].\end{multline} Thus, we may and will
assume that $\alpha^0\equiv W\text{ mod }F_{2g+2}\C[[V^{\vee}]].$

Let $I\subset \C[V^{\vee}]$ be an ideal generated by
$\frac{\partial W}{\partial z_i},$ $i=1,2,3.$ It is easy to see
that
\begin{equation}\label{ideal_I}z_iz_j\in
I+F_{2g}\C[[V^{\vee}]]\text{ for }i<j,\quad z_i^{2g+2}\in I\cdot
F_2\C[[V^{\vee}]]+F_{4g}\C[[V^{\vee}]].\end{equation}

Indeed, for example $z_1z_2\equiv -\frac{\partial W}{\partial
z_3}\text{ mod }F_{2g}\C[[V^{\vee}]],$ and
\begin{equation}z_1^{2g+2}\equiv\frac1{2g+1}z_1^2\frac{\partial W}{\partial z_1}-\frac1{2g+1}z_1z_2\frac{\partial W}{\partial z_2}
-z_2^{2g}\frac{\partial W}{\partial z_3}\text{ mod
}F_{4g}\C[[V^{\vee}]].\end{equation}

Put $W_{4g-1}=\alpha^0.$ It follows from \eqref{conditions} that
$\alpha^0$ contains only monomials of degree $3+(2g-2)k,$ where
$k\geq 0.$ The difference $W-W_{4g-1}$ does not contain monomials
$z_i^{4g-1},$ since they are not $G\text{-}$invariant. It follows
from \eqref{ideal_I} that $W-W_{4g-1}\in I\cdot
F_{4g-3}\C[[V^{\vee}]]+F_{6g-3}\C[[V^{\vee}]].$ Therefore, there
exist homogeneous polynomials $f_{4g-3,1},f_{4g-3,2},f_{4g-3,3}$
of degree $(4g-3),$ such that
\begin{multline}W_{6g-3}=\exp(f_{4g-3,1}\otimes\xi_1+f_{4g-3,2}\otimes\xi_2
+f_{4g-3,3}\otimes\xi_3)^*W_{4g-3}\\
\equiv W_{2g+1}+f_{4g-3,1}\frac{\partial W}{\partial
z_1}+f_{4g-3,2}\frac{\partial W}{\partial
z_2}+f_{4g-3,3}\frac{\partial W}{\partial z_3}\text{ mod }F_{6g-3}\C[[V^{\vee}]]\\
\equiv W\text{ mod }F_{6g-3}\C[[V^{\vee}]].\end{multline}

Moreover, we can take $f_{4g-3,i}$ such that
$(f_{4g-3,1}\otimes\xi_1+f_{4g-3,2}\otimes\xi_2
+f_{4g-3,3}\otimes\xi_3,0)\in \mg^0.$ We obtain a new formal
function $W_{6g-3}\equiv W\text{ mod }F_{6g-3}\C[[V^{\vee}]].$

Now suppose that we are given with some formal function
$W_{3+(2g-2)k},$ where $k\geq 3,$ such that $(W_{3+(2g-2)k},0)\in
\mg^1$ and $W_{3+(2g-2)k}\equiv W\text{ mod
}F_{3+(2g-2)k}\C[[V^{\vee}]].$ It follows from \eqref{ideal_I}
that $W-W_{3+(2g-2)k}\in I\cdot
F_{1+(2g-2)(k-1)}\C[[V^{\vee}]]+F_{3+(2g-2)(k+1)}.$ Thus, there
exist homogeneous polynomials
$f_{1+(2g-2)(k-1),1},f_{1+(2g-2)(k-1),2},f_{1+(2g-2)(k-1),3}$ of
degree $1+(2g-2)(k-1)$ such that
\begin{multline}\exp(f_{1+(2g-2)(k-1),1}\otimes\xi_1+f_{1+(2g-2)(k-1),2}\otimes\xi_2+
f_{1+(2g-2)(k-1),3}\otimes\xi_3)^*W_{3+(2g-2)k}\equiv \\
W\text{
mod }F_{3+(2g-2)(k+1)}.\end{multline} Again, the exponentiated
formal vector field can be taken to belong to $\mg^0.$ We obtain a
new formal function $W_{3+(2g-2)(k+1)},$ such that
$(W_{3+(2g-2)(k+1)},0)\in \mg^1$ and $W_{3+(2g-2)(k+1)}\equiv
W\text{ mod }F_{3+(2g-2)(k+1)}\C[[V^{\vee}]].$

Iterating, we obtain infinite sequence of formal diffeomorphisms,
and their product obviously converges. As a result, our MC
solution $\alpha$ is equivalent to $(W,\alpha'^2)$ for some
$\alpha'^2\in F_{2g}\C[[V^{\vee}]]\otimes \Lambda^2 V.$ Since the
quotient $\C[[V^{\vee}]]/I$ is finite-dimensional, it follows that
the sequence $(\frac{\partial W }{\partial z_1},\frac{\partial W
}{\partial z_2},\frac{\partial W }{\partial z_3})$ is regular in
$\C[[V^{\vee}]],$ and hence the Koszul complex
$\C[[V^{\vee}]]\otimes \Lambda(V)$ with differential $\iota_{dW}$
is a resolution of $\C[[V^{\vee}]]/I.$ Since $\alpha'^2$ is a
cocycle in the Koszul complex, it is also a coboundary. Hence there exists
$\gamma^3\in \C[[V^{\vee}]]\otimes \Lambda^3 V$ such that
$\iota_{dW}(\gamma^3)=-\alpha'^2.$ Again, $\gamma^3$ can be
choosen to belong to $\mg^0.$ By the explicit formula
\eqref{act_3_forms}, the exponential of $(0,\gamma^3)$ maps
$(W,\alpha'^2)$ to $(W,0),$ and we are done.
\end{proof}

\section{Classification theorem on $A_{\infty}$-structures}
\label{classification_theo}

Take the algebra $A=\Lambda(V)$ with standard grading
($\deg(V)=1$). Consider the following DG Lie algebra $\mh:$

\begin{equation}\mh^d \;\; = \!\!\!\!\! \prod_{\substack{3i+j-(4g-4)k = 3d+3 \\
k \geq 0, \, i \geq d+2}}\!\!\!\!\! \Hom^j(A^{\otimes i},A)^G\,
\hbar^k.\end{equation}

The differential is Hochshild differential and the bracket is
Gerstenhaber bracket. Again, $\mh$ is pro-nilpotent with respect
to the filtration $L_{\bullet}\mh,$ where $L_r\mh^d$ is the part
of the product which consists of terms with $i\geq d+1+r.$

Theorem \ref{formality} implies the following lemma (see \cite{Se}
for detailed explanation).

\begin{lemma} There exists a filtered
$L_{\infty}\text{-}$quasi-isomorphism $\Phi:\mh\to \mg,$ with
$\Phi^1$ being the obvious $\hbar\text{-}$linear extension of
Hochshild-Kostant-Rosenberg map.
\end{lemma}

Note that, analogously to the discussion in Section \ref{intro_to_formality}, each MC element $\alpha\in\mh^1$ defines a $\Z/2$-graded
$A_{\infty}$-structure on $A.$ Moreover, equivalent MC elements yield strictly homotopic $A_{\infty}$-structures. In the following sections, on the A side and on the
 B side, we will encounter two different $A_{\infty}$-structures on $A,$ which come from the same equivalence class in $MC(\mh).$

We are going to describe this equivalence class below.

Consider arbitrary $\alpha\in \mh^1.$ Its components are $G$-equivariant $i\text{-}$linear maps
$\alpha^i:A^{\otimes i}\to A,$ for $i\geq 3.$ Further, each $\alpha^i$ has (finite)
decomposition
$\alpha^i=\alpha^i_0+\alpha^i_1\hbar+\alpha^i_2\hbar^2+\dots,$
where \begin{equation}\alpha^i_k\in \Hom^{6-3i+(4g-4)k}(A^{\otimes
i},A)^G.\end{equation}

Note that if $\alpha^i_k\ne 0,$ then $(6-3i+(4g-4)k)\leq 3.$ It
follows that $\alpha^i_1=0$ for $3\leq i<\frac{4g-1}3.$ We will
also need the following elementary observations:

\begin{equation}\label{observation1}L_{2g}\mg^1=(\hbar^2\mg)^1;\end{equation}

\begin{multline}\label{observation2}\Phi^1(\Hom^{2-2g}(A^{\otimes
2g},A)^{G})=(\Sym^{2g}(V^{\vee})\otimes
\Lambda^2(V))^{G}=\\
\begin{cases}\C\cdot z_1^{2g}\otimes
(\xi_2\wedge \xi_3)+\C\cdot z_2^{2g}\otimes (\xi_3\wedge
\xi_1)+\C\cdot z_3^{2g}\otimes (\xi_1\wedge \xi_2) & \text{if }
g\not\equiv 1\text{ mod }3\\
\C\cdot z_1^{2g}\otimes (\xi_2\wedge \xi_3)+\C\cdot
z_2^{2g}\otimes (\xi_3\wedge \xi_1)+\C\cdot z_3^{2g}\otimes
(\xi_1\wedge \xi_2)+\\
\C\cdot
z_1^{\frac{2g+1}3}z_2^{\frac{2g+1}3}z_3^{\frac{2g-2}3}\otimes
(\xi_1\wedge \xi_2)+\C\cdot
z_1^{\frac{2g+1}3}z_2^{\frac{2g-2}3}z_3^{\frac{2g+1}3}\otimes
(\xi_3\wedge \xi_1)+\\
\C\cdot
z_1^{\frac{2g-2}3}z_2^{\frac{2g+1}3}z_3^{\frac{2g+1}3}\otimes
(\xi_2\wedge \xi_3) & \text{if }g\equiv 1\text{ mod
}3;\end{cases}\end{multline}

\begin{multline}\label{observation3} \Phi^1(\Hom^{-2g-1}(A^{\otimes
(2g+1)},A)^{G})=(\Sym^{2g+1}(V^{\vee})^{G}=\\
\begin{cases}\C\cdot z_1^{2g+1}+\C\cdot z_2^{2g+1}+\C\cdot
z_3^{2g+1} & \text{if }g\not\equiv 1\text{ mod }3\\
\C\cdot z_1^{2g+1}+\C\cdot z_2^{2g+1}+\C\cdot z_3^{2g+1}+ \C\cdot
(z_1z_2z_3)^{\frac{2g+1}3} & \text{if }g\equiv 1\text{ mod }3.
\end{cases}
\end{multline}

\begin{theo}\label{classification_theo1} Let $\alpha\in \mh^1$ be an MC element such that
$\Phi^1(\alpha^3_0)=-z_1z_2z_3$ and
\begin{equation}\Phi^1(\alpha^{2g+1}_1)=\begin{cases}z_1^{2g+1}+z_2^{2g+1}+z_3^{2g+1} & \text{ if }g\not\equiv 1\text{ mod }3\\
z_1^{2g+1}+z_2^{2g+1}+z_3^{2g+1}+\lambda
(z_1z_2z_3)^{\frac{2g+1}3},\text{ where }\lambda\in \C & \text{ if
}g\equiv 1\text{ mod }3.\end{cases}\end{equation}
Then we have that MC element $\Phi_*(\alpha)\in MC(\mg)$ is equivalent to $(W,0)\in MC(\mg),$ in the notation of the previous section.
In particular, all such $\alpha$ are equivalent to each other.
\end{theo}

\begin{proof} First we will replace $\alpha$ with equivalent
$\alpha'$ satisfying the assumptions of the theorem, and such that
$\alpha'^i_1=0$ for $3\leq i<2g.$ We will need the following
\begin{lemma}\label{action_of_exp} 1) Take some $\gamma^i_1\in \mh^0$ lying in the component
$\Hom^{3-3i+(4g-4)}(A^{\otimes i},A).$ Then for each MC element
$\alpha\in \mh^1$ we have
\begin{equation}\alpha'=\exp(\gamma^i_1)\cdot \alpha\equiv
\alpha-\partial\gamma+[\gamma,\alpha]\text{ mod
}(\hbar^2\mh)^1.\end{equation}

2) If. moreover, $i\leq 2g-2,$ then we have that $\alpha'$
satisfies the assumptions of the theorem.
\end{lemma}
\begin{proof} 1) This is evident.

2) According to 1) and \eqref{observation1}, we only need to check
that the polynomial $\Phi^1([\gamma^i_1,\alpha^{2g+2-i}_0])$ does
not contain monomials $z_i^{2g+1}.$ But for degree reasons, for
$2\leq i\leq 2g-2$ we have that $\alpha^{2g+2-i}_0$ vanishes when
restricted to $V^{\otimes (2g+2-i)}.$ Further, for $2\leq i\leq
2g-3,$ we have that $\gamma^i_1$ vanishes when restricted to
$V^{\otimes i}.$ Therefore, in the case $2\leq i\leq 2g-3$
$[\gamma^i_1,\alpha^{2g+2-i}_0]$ vanishes on $V^{\otimes (2g+1)},$
hence the assertion.

Further, in the case $i=2g-2,$ it suffices to notice that
$\gamma_1^{2g-2}(\xi_i^{\otimes 2g-2})=0$ from the
$G\text{-}$equivariance condition.
\end{proof}





Take the smallest $i_0$ such that $\alpha^{i_0}_1\ne 0.$ Suppose
that $i_0<2g.$ Since $\alpha$ is MC solution, we have that
$\partial\alpha^{i_0}_1=0.$ Denote by $\bar{A}=\sum\limits_{k\geq
1}\Lambda^k(V)$ the augmentation ideal of $A.$ Simple degree
counting shows that $\Hom^{6-3i_0+4g-4}(\bar{A}^{\otimes
i_0},A)=0.$ Since the reduced Hochshild complex embeds
quasi-isomorphically to the standard one, we have that there
exists $\gamma^{i_0-1}_1\in \mh^0$ such that $\partial
\gamma^{i_0-1}_1=\alpha^{i_0}_1.$ Then, it follows from Lemma
\ref{action_of_exp} that $\alpha'=\exp(\gamma^{i_0-1}_1)\cdot \alpha$
satisfies the assumptions of the theorem. Moreover,
$\alpha'^i_1=0$ for $3\leq i\leq i_0.$

Iterating, we obtain some equivalent MC solution $\alpha'\in
\mh^1$ satisfying the assumptions of the theorem and such that
$\alpha'^i_1=0$ for $3\leq i<2g.$ Assume from this moment that
$\alpha$ itself satisfies this property.

Since $\alpha$ is MC solution, we have
\begin{equation}\partial\alpha^3_0=0,\quad\partial\alpha^{2g}_1=0,\quad\partial\alpha^{2g+1}_1+[\alpha^3_0,\alpha^{2g}_1]=0.\end{equation}

Therefore, $\alpha^{2g}_1$ satisfies the identity
\begin{equation}\label{cond_on_Phi^1(alpha^2g_1)}[z_1z_2z_3,\Phi^1(\alpha^{2g}_1)]=-[\Phi^1(\alpha^3_0),\Phi^1(\alpha^{2g}_1)]=
-\Phi^1([\alpha^3_0,\alpha^{2g}_1])=\Phi^1(\partial
\alpha^{2g+1}_1)=0.\end{equation}

From \eqref{cond_on_Phi^1(alpha^2g_1)} and from
\eqref{observation2} we conclude that
\begin{multline}\label{Phi^1(alpha^2g_1)}\Phi^1(\alpha^{2g}_1)=\begin{cases}0 & \text{if }g\not\equiv 1\text{ mod }3\\
\lambda'(z_1^{\frac{2g+1}3}z_2^{\frac{2g+1}3}z_3^{\frac{2g-2}3}\otimes
(\xi_1\wedge \xi_2)+
z_1^{\frac{2g+1}3}z_2^{\frac{2g-2}3}z_3^{\frac{2g+1}3}\otimes
(\xi_3\wedge \xi_1)+\\
z_1^{\frac{2g-2}3}z_2^{\frac{2g+1}3}z_3^{\frac{2g+1}3}\otimes
(\xi_2\wedge \xi_3)),\,\lambda'\in \C & \text{if }g\equiv 1\text{
mod }3.\end{cases}
\end{multline}

Simple degree counting shows that
\begin{multline}\label{tildealpha}\tilde{\alpha}:=\sum\limits_{n\geq 1}(-1)^{\frac{n(n-1)}{2}}
\frac{1}{n!}\Phi^n(\alpha,\dots,\alpha)\equiv
\Phi^1(\alpha^3_0)+\hbar
\Phi^1(\alpha^{2g+1}_1)-\hbar\Phi^2(\alpha^3_0,\alpha^{2g}_1)\\
\text{ mod }L_{2g}\mg^1=(\hbar^2 \mg)^1.\end{multline}

\begin{lemma}\label{no_changes}The polynomial $\Phi^2(\alpha^3_0,\alpha^{2g}_1)\in \Sym^{2g+1}(V^{\vee})$
does not contain terms $z_i^{2g+1}.$\end{lemma}
\begin{proof}
If $\alpha'^{2g}_1\in \Hom^{2-2g}(A^{\otimes 2g},A)$ is a
Hochshild cocycle homologous to $\alpha^{2g}_1$ and
$\gamma^2_0\in\Hom^{-3}(A^{\otimes 2},A),$ then
\begin{multline*} \Phi^2(\partial \gamma^2_0,\alpha'^{2g}_1)=\pm
\Phi^2(\gamma^2_0,\partial \alpha'^{2g}_1)\pm
\Phi^1([\gamma^2_0,\alpha'^{2g}_1])\pm
\partial\Phi^2(\gamma^2_0,\alpha'^{2g}_1)\pm
[\Phi^1(\gamma^2_0),\Phi^1(\alpha'^{2g}_1)]\\
=\pm \Phi^1([\gamma^2_0,\alpha'^{2g}_1]).
\end{multline*}
It follows from \eqref{Phi^1(alpha^2g_1)} that the RHS of the
above chain of identities does not contain monomials $z_i^{2g+1}.$
Analogously, if $\alpha'^3_0\in\Hom^{-3}(A^{\otimes 3},A)$ is a
Hochshild cocycle homologous to $\alpha^3_0$ and
$\gamma^{2g-1}_1\in \Hom^{2-2g}(A^{\otimes (2g-1)},A),$ then we
have that $\Phi^2(\alpha'^3_0,\partial\gamma^{2g-1}_1)$ does not
contain terms $z_i^{2g+1}.$ Therefore, we may assume that
\begin{equation}\alpha^3_0=\Psi^1\Phi^1(\alpha^3_0),\quad \alpha^{2g}_1=
\Psi^1\Phi^1(\alpha^{2g}_1),\end{equation} where $\Psi:\mg\to \mh$
is (the obvious $\hbar\text{-}$linear extension of) Kontsevich's
$L_{\infty}\text{-}$quasi-isomorphism. Further,
$L_{\infty}\text{-}$morphism $\Phi$ can be taken to be strictly
left inverse to $\Psi,$ that is $\Phi\Psi=\Id$ (Remark
\ref{left_inverse}). Under this assumptions, the coefficients of
$\Phi^2(\alpha^3_0,\alpha^{2g_1})$ in the monomials $z_i^{2g+1}$
equal to
\begin{equation}\label{Psi^2}\pm\Psi^2(\Phi^1(\alpha^3_0),\Phi^1(\alpha^{2g}_1))(\xi_i^{\otimes (2g+1)}),\quad i=1,2,3.\end{equation}
From the precise formulas for $\Phi^1(\alpha^3_0)$($=-z_1z_2z_3$)
and $\Phi^1(\alpha^{2g}_1)$ (formula \eqref{Phi^1(alpha^2g_1)}),
as well as for the component $\Psi^2$ (\cite{Ko}, subsection 6.4,
with suitable changes) one obtains that \eqref{Psi^2} equals to
zero, as follows. In the notation of \cite{Ko}, subsection 6.4,
for each relevant admissible graph $\Gamma$ we have
$\mathcal{U}_{\Gamma}(\Phi^1(\alpha^3_0),\Phi^1(\alpha^{2g}_1))(\xi_i^{\otimes
(2g+1)})=0.$ Since $\Psi^2$ is a linear combination of
$\mathcal{U}_{\Gamma},$ we obtain that \eqref{Psi^2} equals to
zero.
\end{proof}

Further, $L_{2g}\mg^1=(\hbar^2 \mg)^1$ consists of pairs
$(\tilde{\alpha}^0,\tilde{\alpha}^2)$ such that
$\tilde{\alpha}^0\in F_{4g-1}\C[[V^{\vee}]],$ and
$\tilde{\alpha}^2\in F_{4g-2}\C[[V^{\vee}]]\otimes \Lambda^2 V.$
From \eqref{tildealpha} and Lemma \ref{no_changes} it follows that
$\tilde{\alpha}$ satisfies the assumptions of Lemma
\ref{equiv_to_(W,0)}. Therefore, $\tilde{\alpha}$ is equivalent to
$(W,0).$ By Lemma \ref{bijection}, $\Phi$ induces a bijection on
the equivalence classes of Maurer-Cartan solutions. It follows
that $\alpha$ with required properties is unique up to equivalence.
\end{proof}

We are interested in the following reformulation of the above
Theorem. Suppose that we are given with a $(\Z/2)\text{-}$graded
$A_{\infty}\text{-}$structure $(\mu^1,\mu^2,\dots)$ on
$A=\Lambda(V).$ Moreover, assume that all $\mu^i$ are
$G\text{-}$equivariant, $\mu^1=0,$ $\mu^2$ is the usual wedge
product (twisted by sign), and for $i\geq 3$ we have (finite) decomposition
$\mu^i=\mu^i_0+\mu^i_1+\dots,$ where $\mu^i_k$ is homogeneous of
degree $6-3i+(4g-4)k$ with respect to $\Z\text{-}$gradings.
Suppose that for $z\in V\subset A$ we have
\begin{equation}\mu^3_0(z,z,z)=-z_1z_2z_3,\end{equation}
and
\begin{equation}\mu^{2g+1}_1(z,\dots,z)=\begin{cases}z_1^{2g+1}+z_2^{2g+1}+z_3^{2g+1} & \text{if }g\not\equiv 1\text{ mod }3\\
z_1^{2g+1}+z_2^{2g+1}+z_3^{2g+1}+\lambda(z_1z_2z_3)^{\frac{2g+1}3},\,\lambda\in\C
& \text{if }g\equiv 1\text{ mod }3.\end{cases}\end{equation} Then by Theorem \ref{classification_theo1}
such a structure is determined uniquely up to
$G\text{-}$equivariant $A_{\infty}\text{-}$quasi-isomorphisms. We
denote this class of $G\text{-}$equivariant
$A_{\infty}\text{-}$structures by $\cA'.$

\section{Categories of singularities and matrix factorizations}
\label{matr_fact}

Let $V=\C^n$ and take some non-zero polynomial $W\in \C[V^{\vee}]$ such
that the hypersurface $W^{-1}(0)$ has (not necessarily isolated)
singularity at the origin. Following Orlov \cite{Or}, associate to it the
triangulated category of singularities:
\begin{equation}D_{sg}(W^{-1}(0))=D^b_{coh}(W^{-1}(0))/\Perf(W^{-1}(0)).\end{equation}

Denote by $\overline{D_{sg}}(W^{-1}(0))$ the idempotent completion
of $D_{sg}(W^{-1}(0)).$ The following Lemma easily follows from
the results in \cite{Or2} (see \cite{Se}, proof of Lemma 12.1):

\begin{lemma}\label{O_0_generates} If $W$ has the only singular point at the origin, then the triangulated category $\overline{D_{sg}}(W^{-1}(0))$ is split-generated by
the image of the structure sheaf $\cO_{0}.$\end{lemma}

It turns out that the triangulated category $D_{sg}(W^{-1}(0))$ is
$(\Z/2)\text{-}$graded, i.e. the shift by $2$ in
$D_{sg}(W^{-1}(0))$ is canonically isomorphic to the identity
(this follows from Theorem \ref{Orlov's_theorem} below).

Now we define the D$(\Z/2)\text{-}$G category $MF(W)$ of matrix factorizations of $W.$  Matrix
factorizations give a $(\Z/2)\text{-}$graded enhancement of this
category. A matrix factorization for $W$ is a pair of projective (hence free) finitely generated $\C[V^{\vee}]\text{-}$modules $(E^0,E^1),$
together with a pair of morphisms $\delta_E^1:E^1\to E^0,$ $\delta_E^0:E^0\to E^1,$ such that
\begin{equation}\delta_E^1\delta_E^0=W\cdot\id_{E^0},\quad \delta_E^0\delta_E^1=W\cdot\id_{E^1}.\end{equation}
In particular, $E^0$ and $E^1$ have the same rank. Denote by $E=E^0\oplus E^1$ the $\Z/2$-graded $\C[V^{\vee}]$-module,
 and $\delta_E=\delta_E^0\oplus \delta_E^1:E\to E$ the corresponding odd map. We call the map $\delta_E$ "differential"
, although its square does not equal to zero.

If $(E,\delta_E)$ and $(F,\delta_F)$ are matrix factorizations, then we have $2$-periodic complex of morphisms $\Hom(E,F).$ Namely,
as a $\Z/2$-graded  vector space, it consists of all even and odd maps of $\Z/2$-graded modules. The differential is a super-commutator with $\delta.$ It is easy to see that $MF(W)$ is a
strongly pre-triangulated D$(\Z/2)\text{-}$G
category.

\begin{theo}\label{Orlov's_theorem}(\cite{Or}, Theorem 3.9) There is a natural exact equivalence of triangulated categories
$\Ho(MF(W))\sim D_{sg}(W^{-1}(0)).$ \end{theo}

This equivalence associates to a matrix factorization
$(E,\delta_E)$ the projection of $\coker(\delta^1:E^1\to E^0)$
(clearly, $W$ annihilates this $\C[V^{\vee}]\text{-}$module, hence
it can be considered as an object of $D^b_{coh}(W^{-1}(0))$).

We would like to write down explicitly the matrix factorization which corresponds to the structure sheaf of origin under the equivalence of Theorem
\ref{Orlov's_theorem}.
Decompose the polynomial $W$ into the sum of its graded
components:
\begin{equation}W=\sum\limits_{i=2}^kW_i,\quad W_i\in\Sym^i(V^{\vee}).\end{equation}
Take the one-form
\begin{equation}\gamma=\sum\limits_{i=2}^k\frac{dW_i}{i}.\end{equation} Denote by $\eta=\sum\limits
z_k\xi_k$ the Euler vector field on $V.$

Now take the matrix factorization $(E,\delta_E)$ with
$E=\Omega(V)=\C[V^{\vee}]\otimes \Lambda(V^{\vee}),$ and $\delta_E=\iota_{\eta}+\gamma\wedge\cdot.$ It
is easy to see that $\delta_E^2=\gamma(\eta)\cdot\id=W\cdot\id.$

\begin{lemma}\label{Mf_for_O_0}(\cite{Se}, Lemma 12.3) The object $\coker(\delta_E^1)$ is isomorphic to
$\cO_0$ in $D_{sg}(W^{-1}(0)).$
\end{lemma}

\begin{remark}In a similar way, one can write down matrix factorization, corresponding to $\cO_Z,$ where $Z\subset W^{-1}(0)$ is any closed subscheme, which is complete intersection in $V.$\end{remark}

Take the D$(\Z/2)\text{-}$G algebra
\begin{equation}\label{B_W}\cB_W:=\End_{MF(W)}(E).\end{equation} By Lemma \ref{Mf_for_O_0}, it is quasi-isomorphic to
the D$(\Z/2)\text{-}$G algebra
$\bR\Hom_{D_{sg}(W^{-1}(0))}(\cO_0,\cO_0).$  We have the following

\begin{cor}\label{Perf(B_W)} Suppose that $W$ has the only singular point at the origin. Then there is an equivalence $\overline{D_{sg}}(W^{-1}(0))\cong \Perf(\cB_W).$\end{cor}

\section{Minimal $A_{\infty}$-model for $\cB_W$}
\label{Koszul}

In this section we describe more explicitly the DG algebra $\cB_W$
introduced in \eqref{B_W}. We also prove that in the special case of
our LG model, it is (equivariantly) quasi-isomorphic to the
$A_{\infty}\text{-}$algebra $\cA'$ from the end of section
\ref{classification_theo} (Proposition \ref{A_equiv_B})

Let $V=\C^n.$ Consider $\Omega(V)=\C[V^{\vee}]\otimes
\Lambda(V^\vee)$ as a complex of $\C[V^{\vee}]\text{-}$modules
with $\deg(\C[V^{\vee}]\otimes \Lambda^kV^\vee)=-k$ and
differential $\iota_{\eta},$ where $\eta=\sum\limits_{k=1}^n
z_k\xi_k$ is the Euler vector field. This complex is just a Koszul resolution of the structure sheaf of the origin $\cO_0$.

Consider the DG algebra
$B=\End_{\C[V^{\vee}]}(\Omega(V)).$ We have that $H^{\cdot}(B)\cong \Ext_{\C[V^{\vee}]}^{\cdot}(\cO_0,\cO_0)\cong
\Lambda(V).$ Further, we can identify
\begin{equation}B\cong \Omega(V)\otimes \Lambda(V),\end{equation}
where for $f\in \Sym(V^{\vee}),$ $\beta\in \Lambda(V^{\vee}),$
$\theta\in \Lambda(V)$ the element $f\beta\otimes \theta\in
\Omega(V)\otimes\Lambda(V)$ corresponds to the endomorphism
$f\beta\wedge \iota_{\theta}(\cdot)\in
B=\End_{\C[V^{\vee}]}(\Omega(V)).$

Explicitly, the differential $\partial:\Omega(V)\otimes
\Lambda(V)\to \Omega(V)\otimes \Lambda(V)$ is given by the formula
\begin{equation}\partial(f\beta\otimes\theta)=\iota_{\eta}(f\beta)\otimes\theta.\end{equation} It is well known that DG
algebra $B$ is formal. Moreover, we can write down explicitly the
quasi-isomorphism of DG algebras $i:\Lambda(V)\to B,$
\begin{equation}i(\theta)=1\otimes \theta.\end{equation} Also, consider the natural
projection $p:B\to\Lambda(V),$

\begin{equation}\begin{cases}p(1\otimes\theta)=\theta & \text{for }\theta\in \Lambda(V);\\
p(f\beta\otimes\theta)=0 & \text{for }f\in \Sym^r(V^{\vee}), \beta\in \Lambda^s(V^{\vee}),
\theta\in \Lambda(V), r+s>0.\end{cases}\end{equation} Clearly, $pi=\id_{\Lambda(V)}.$
Further, $ip$ differs from $\id_B$ by homotopy given by the fromula
\begin{equation}h(f\beta\otimes \theta)=\begin{cases}0 & \text{if }f\beta=\lambda,\,\lambda\in\C\\
\frac1w (df\wedge\beta)\otimes\theta &
\text{otherwise,}\end{cases}\end{equation} where $w=r+s,$ $f\in
\Sym^r(V^{\vee}),$ $\beta\in \Lambda^s(V^{\vee}).$ Moreover, the
maps $h,$ $p,$ $i$ satisfy the following identities:
\begin{equation}h^2=0,\quad ph=0,\quad hi=0.\end{equation}

Now take the polynomial $W\in \C[V^{\vee}]$ with singularity at
the origin. In the previous section we have written down the one-form $\gamma\in\Omega^1(V),$ such that $\iota_{\eta}(\gamma)=W.$
Such $\gamma$ defines a matriz factorization $E=(\Omega(V),\iota_{\eta}+\gamma\wedge\cdot).$ We defined the
D$(\Z/2)$-G algebra $\cB_W:=\End(E).$ It is clear that $\cB_W^{gr}\cong B^{gr},$ where
$\cB_W^{gr}$ (resp. $B^{gr}$) is the underlying
$(\Z/2)\text{-}$graded algebra of $\cB_W$ (resp. $B$). Denote the
differential on $\cB_W$ by $\tilde{\partial}.$ We have the
following explicit formula for the difference of differentials:
\begin{equation}(\tilde{\partial}-\partial)(f\beta\otimes
\theta)=(-1)^{|\beta|-1}\sum\limits_{k=1}^n g_k f\beta\otimes
\iota_{dz_k}\theta,\text{ where }\gamma=\sum\limits_{k=1}^n
g_kdz_k.\end{equation}

We are going to describe the minimal $A_{\infty}$-model for $\cB_W.$ It is obtained from the maps $h,p,i$ above using standard formula of summing up over trees. We
obtain a $(\Z/2)\text{-}$graded $A_{\infty}\text{-}$structure
$\cA$ on the $(\Z/2)$-graded vector space $A=\Lambda(V)$ together with
$A_{\infty}\text{-}$quasi-isomorphism $\cA\to \cB.$ Explicit
computation of $\mu^k:A^{\otimes k}\to A$ goes as follows.
Consider a ribbon tree with $(k+1)$ semi-infinite edges, $k$
incoming and one outgoing, which has only bivalent and trivalent
vertices. Associate with each vertex and each edge an operation by the following formulas:
\begin{equation}\begin{cases}\text{for a bivalent vertex} & b\mapsto
(-1)^{|b|}(\tilde{\partial}-\partial)(b),\,\cB\to\cB;\\
\text{for a trivalent vertex} & (b_2,b_1)\mapsto
(-1)^{|b_1|}b_2b_1,\,\cB^{\otimes 2}\to\cB;\\
\text{for a finite edge} & b\mapsto (-1)^{|b|-1}h(b),\,\cB\to\cB;\\
\text{for an incoming edge} & a\to i(a),\,A\to \cB;\\
\text{for an outgoing edge} & b\to p(b),\,\cB\to
A.\end{cases}\end{equation} Then each such tree gives a map
$A^{\otimes k}\to A$ in an obvious way. The explicit expression for $\mu^k:A^{\otimes k}\to A$ is
just the sum of contributions of all possible trees. The sum is
actually finite because
\begin{equation}(\tilde{\partial}-\partial)(C[[V^{\vee}]]\otimes\Lambda^k(V^{\vee})\otimes\Lambda(V))\subset
C[[V^{\vee}]]\otimes\Lambda^k(V^{\vee})\otimes\Lambda(V),\text{
and}\end{equation}
\begin{equation}h(C[[V^{\vee}]]\otimes\Lambda^k(V^{\vee})\otimes\Lambda(V))\subset
C[[V^{\vee}]]\otimes\Lambda^{k+1}(V^{\vee})\otimes\Lambda(V).\end{equation}

The components $f_k:\cA^{\otimes k}\to \cB$ of the
$A_{\infty}\text{-}$quasi-isomorphism are defined in the same way
with the only difference: to the outgoing edge one attaches the
operation $b\to h(b).$ Again, the sum over trees is actually
finite.

To see that $f_1$ is quasi-isomorphism, take the increasing
filtrations by subcomplexes:

\begin{equation}F_r\cB_W=\Omega(V)\otimes \Lambda^{\leq r}(V),\quad F_r\Lambda(V)=\Lambda^{\leq r}(V).\end{equation}
Then the map $f_1:\Lambda(V)\to \cB_W$ is compatible with these
filtrations, and it induces quasi-isomorphisms on the
subquotients.

Return to the special case $V=\C^3,$
$W=-z_1z_2z_3+z_1^{2g+1}+z_2^{2g+1}+z_3^{2g+1}.$ Then we have
\begin{equation}\label{g_k}g_1=-\frac{z_2z_3}3+z_1^{2g},\quad
g_2=-\frac{z_1z_3}3+z_2^{2g},\quad
g_3=-\frac{z_1z_2}3+z_3^{2g}.\end{equation}

\begin{prop}\label{A_equiv_B} In the above notation, the resulting $A_{\infty}\text{-}$algebra $\cA$ is
$G\text{-}$equivariantly equivalent to $\Lambda(V)$ with the
$A_{\infty}\text{-}$structure $\cA'$ from the end of section
\ref{classification_theo}.\end{prop}

\begin{proof} It is useful to take the following $\Z\text{-}$grading on $B=\Omega(V)\otimes \Lambda(V).$
\begin{equation}\deg(\Sym^i(V^{\vee})\otimes\Lambda^j(V^{\vee})\otimes\Lambda^k(V))=2i-j+k.\end{equation}
Then $\partial$ has degree $3,$ $h$ has degree $-3.$ If we want
$\tilde{\partial}$ to have degree $3,$ we should introduce a
formal parameter $\hbar$ with degree $(4-4g).$ Further, we should
write $g_1=-\frac{z_2z_3}3+\hbar z_1^{2g}$ and analogously for
other $g_i.$ The operations $\mu^d$ are then decomposed as
follows: $\mu^d=\mu^d_0+\mu^d_1\hbar+\mu^d_2\hbar^2+\dots,$ with
$\mu^d_k$ being of degree $(6-3d+(4g-4)k).$ Also, it is easy to
see that all $\mu^d$ are $G\text{-}$equivariant. It is
straightforward to check that $\mu^1_{\cA}=0,$ and $\mu^2_{\cA}$
the usual wedge product (this follows from vanishing of the degree
$2$ component of $W$). Further, the only tree (see Figure 1)
contributes to $\Phi^1(\mu^3_0),$ and it equals to $-z_1z_2z_3.$
Analogously, the only tree (see Figure 2) contributes to
$\Phi^1(\mu^{2g+1}_1),$ and it equals to
$z_1^{2g+1}+z_2^{2g+1}+z_3^{2g+1},$ as prescribed. This proves
Proposition.\end{proof}

\begin{center}\begin{tabular}{c}\epsfbox{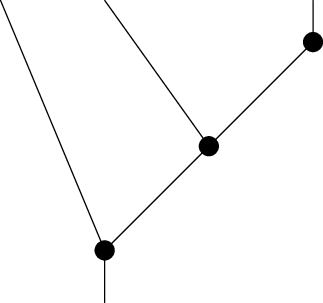}\\
\LARGE{Figure 1.}
\end{tabular}\hfil\begin{tabular}{c}\epsfbox{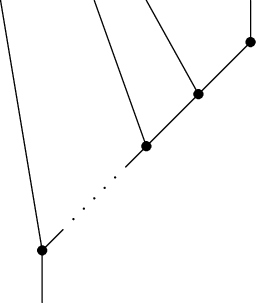}\\\LARGE{Figure 2.}\end{tabular}\end{center}

From Corollary \ref{Perf(B_W)} and Proposition \ref{A_equiv_B} we
obtain the equivalence
\begin{equation}\overline{D_{sg}}(W^{-1}(0))\cong \Perf(\cA').\end{equation}

Further, Orlov's theorem can be extended to the equivariant
setting. Let $K\subset G$ be the cyclic subgroup of order $2g+1,$
generated by the diagonal matrix with diagonal entries
$(\zeta,\zeta,\zeta^{2g-1}),$ where $\zeta$ is the primitive
$(2g+1)\text{-}$th root of unity. Then $D_{sg,K}(W^{-1}(0))$ is
equivalent to $MF_{K}(W).$ The projection of $\cO_0\otimes \C[K]$
split-generates $D_{sg,K}(W^{-1}(0)).$ In $K\text{-}$equivariant
matrix factorizations it corresponds to $(\Omega(V)\otimes
\C[K],\iota_{\eta}+\gamma\wedge\cdot).$ Its endomorphism DG
algebra is naturally isomorphic to the smash product $\C[K]\#
\cB_W,$ which is further $A_{\infty}\text{-}$quasi-isomorphic to
$\C[K]\# \cA'.$ The result is

\begin{cor}\label{A'*Z_{2g+1}}The category $\overline{D_{sg,K}}(W^{-1}(0))$ is equivalent to $\Perf(\C[K]\# \cA').$\end{cor}

\section{Reconstruction theorem}
\label{sing_general}

The results of this section will not be used in the proof of main
theorem.

Here we show that one can recover the polynomial $W$ (up to formal
change of variables) from the $A_{\infty}\text{-}$structure on
$\Lambda(V)$ transferred from D$(\Z/2)\text{-}$G algebra $\cB_W,$
as in the previous section, for general $W.$ Our proof is based on
Kontsevich formality theorem, and on Keller's paper \cite{Ke1}.

More precisely, our setting is the following. Let $\mk$ be any
field of characteristic zero and $V=\mk^n,$ $n\geq 1.$ Consider a
polynomial $W=\sum\limits_{i=3}^r W_r\in \mk[V^{\vee}],$ with
$W_i\in \Sym^i(V^{\vee}).$ Take the D$(\Z/2)\text{-}$G algebra
$\cB_W.$ We have the canonical isomorphism of super-algebras
\begin{equation}\label{exterior}\Lambda(V)\cong H^{\cdot}(\cB_W).\end{equation}

\begin{theo}\label{reconstr}Let $W,$ $W'$ be non-zero polynomials as above. Suppose that DG algebras $\cB_{W}$ and $\cB_{W'}$ are quasi-isomorphic, and the chain of quasi-isomorphisms
connecting $\cB_{W}$ with $\cB_{W'}$ induces the identity in
cohomology via identifications \eqref{exterior}. Then $W'$ can be
obtained from $W$ by a formal change of variables of the form
\begin{equation}\label{var_change}z_i\to z_i+O(z^2).\end{equation}
\end{theo}
\begin{proof} We introduce four pro-nilpotent DG algebras. First
define the DGLA $\widetilde{\mg}$ by the formula
\begin{equation}\widetilde{\mg}^d=\prod_{\substack{j-2k=d+1 \\ k\in\Z,\, i\geq d+2}}(\Sym^i(V^{\vee})\otimes \Lambda^j(V))\cdot\hbar^k,
\end{equation}
and $L_r\widetilde{\mg}^d$ is the part of the product with $i\geq
d+1+r,$ $r\geq 1$ (the differential is zero, and the bracket is
Schouten one). Further, put
\begin{equation}\widetilde{\mh_1}^d=\prod_{\substack{i+j-2k=d+1 \\ k\in\Z,\, i\geq d+2}}(\Hom^j(\Lambda(V)^{\otimes i},\Lambda(V))\cdot\hbar^k,
\end{equation}
and $L_r\widetilde{\mh_1}^d$ is the part with $i\geq d+1+r$ (the
differential is Hochshild one and the bracket is Gerstenhaber
one). Now, take the "lower" grading on $\mk[[V^{\vee}]],$ with
$\mk[[V^{\vee}]]_d=\Sym^d(V^{\vee}).$ Of course, $\mk[[V^{\vee}]]$
is the {\it direct product} of its graded components, but {\it not
direct sum}. For the rest of this section we will denote the
standard grading by upper indices, and the "lower" grading by the
lower indices. Define the DGLA $\widetilde{\mh_2}$ by the formula
\begin{equation}\widetilde{\mh_2}^d=\prod_{\substack{i-2k=d+1 \\ k\in\Z,\, i\geq 0,\, j'+2k\geq 1.}}(\Hom_{j'}(\mk[[V^{\vee}]]^{\otimes i},\mk[[V^{\vee}]])\cdot\hbar^k,
\end{equation}
with $L_r\widetilde{\mh_2}^d$ being the part of the product with
$j'+2k\geq r.$

Now take the Koszul DG
$\mk[[V^{\vee}]]\text{-}\Lambda(V)\text{-}$bimodule
$X=\Lambda(V^{\vee})\otimes \mk[[V^{\vee}]]$ with the "upper" and
"lower" gradings
$X^{j}_{j'}=\Lambda^{-j}(V^{\vee})\otimes\Sym^{j'}(V^{\vee}),$ and
with differential $\iota_{\eta}$ of bidegree $(1,1).$ Define the
DGLA $Q$ by the formula
 \begin{equation}Q^d=\widetilde{\mh_1}^{d}\oplus \widetilde{\mh_2}^{d}\oplus \prod_{\substack{i_1+i_2+j-2k=d\\ 2k+j'-j\geq 1}}\Hom^j_{j'}
 (\Lambda(V)^{\otimes i_1}\otimes X\otimes \mk[[V^{\vee}]]^{\otimes i_2},X)\cdot\hbar^k,\end{equation}
where the differential and the bracket are induced by those in the
Hochshild complex of the DG category $\cC,$ where

- $Ob(\cC)=\{Y_1,Y_2\};$

- $\Hom_{\cC}(Y_1,Y_1)=\mk[[V^{\vee}]];$

- $\Hom_{\cC}(Y_2,Y_2)=\Lambda(V);$

- $\Hom_{\cC}(Y_1,Y_2)=X;$

- $\Hom_{\cC}(Y_2,Y_1)=0,$

Composition law in $\cC$ comes from the bimodule structure on $X$
(and from algebra structures on $\mk[[V^{\vee}]],$ $\Lambda(V)$).
Thus, the DGLA structure on $Q$ is defined. Further, define
\begin{equation}L_rQ^d=L_r\widetilde{\mh_1}^{d}\oplus L_r\widetilde{\mh_2}^{d}\oplus(\text{part of the product with }2k+j'-j\geq r),\quad r\geq 1.\end{equation}

It follows from \cite{Ke1}, Lemma in Subsection 4.5, that natural
projections $p_i:Q\to \widetilde{h_i},$ $i=1,2,$ are
quasi-isomorphisms of DGLA's. Moreover, both $p_1,p_2$ are
filtered quasi-isomorphisms, as it is straightforward to check.

According to \cite{Ko2}, one can attach to all Kontsevich
admissible graphs (relevant for the formality theorem) {\it
rational} weights, in such a way that they give formality
$L_{\infty}\text{-}$quasi-isomorphism (i.e. satisfy the relevant
system of quadratic equations). In this way we obtain filtered
$L_{\infty}\text{-}$quasi-isomorphism $\cU:\widetilde{\mg}\to
\widetilde{\mh_2}.$

Since $p_1,p_2,\cU$ are filtered
$L_{\infty}\text{-}$quasi-isomorphisms, we have by Lemma
\ref{bijection} that the composition $p_{1*}\circ
(p_{2*})^{-1}\circ \cU_*:MC(\widetilde{\mg})\to
MC(\widetilde{\mh_1})$ is a bijection between the sets of
equivalence classes of MC solutions in $\widetilde{\mg}$ and
$\widetilde{\mh_1}.$

To prove the theorem, we need to prove that, under the assumptions
of the theorem, MC equations $W,W'\in \widetilde{\mg}^1$ are
equivalent. Indeed, this means that $W'$ is the pullback of $W$
under the formal diffeomorphism of $V$ with zero differential at
the origin. Therefore, it suffices to prove the following

\begin{lemma} Under the above bijection between equivalence classes of MC solutions, the class of $W\in \widetilde{\mg}^1$ corresponds
to the class $\alpha\in \widetilde{\mh_1}^1$ of the
$(\Z/2)\text{-}$graded) $A_{\infty}\text{-}$structure on
$\Lambda(V)$ transferred from $\cB_W$ to $H^{\cdot}(\cB_W)\cong
\Lambda(V).$
\end{lemma}
\begin{proof} First note that $\cU^k(W,\dots,W)=0$ for $k>1,$ and
$\cU^1(W)$ has the only constant component which is equal to $W.$

Denote by $\mu=(\mu^3,\mu^4,\dots)$ the
$A_{\infty}\text{-}$structure on $\Lambda(V)\cong
H^{\cdot}(\cB_W)$ transferred from $\cB_W,$ as in the previous
section. Let $\cA$ be the resulting $A_{\infty}\text{-}$algebra.
Denote by $f=(f_1,f_2\dots)$ the
$A_{\infty}\text{-}$quasi-isomorphism $\cA\to\cB_W.$ Also denote
by $f_0\in \cB_W^1$ the multiplication by the $1\text{-}$form
$\gamma.$ We can consider $f_i$ as maps $f_i:A^{\otimes i}\otimes
X\to X.$ Now define $\tilde{\alpha}\in Q^1$ with components
$\mu^i,$ $i\geq 3,$ $f_j,$ $j\geq 0,$ and $W\in
\widetilde{\mh_2}^1.$ Then $\tilde{\alpha}$ is MC solution,
\begin{equation}p_1(\tilde{\alpha})=\alpha,\text{ and }
p_2(\tilde{\alpha})=\cU^1(W)=\sum\limits_{k\geq
1}(-1)^{\frac{k(k-1)}{2}}\frac1{k!}\cU^k(W,\dots,W).\end{equation} Thus, classes of MC
solutions $W\in \widetilde{\mg}^1$ and $\alpha\in
\widetilde{\mh}^1$ correspond to each other. Lemma is proved.
\end{proof}
Theorem is proved.
\end{proof}

It follows from the proof of the above Theorem that there exists
filtered $L_{\infty}\text{-}$morphism
$\tilde{\Phi}:\widetilde{\mh_1}\to\widetilde{\mg}$ such that the
polynomial $W$ can be reconstructed from $\cB_W$ as follows. Take
$\alpha\in \widetilde{\mh_1}^1$ to be MC solution corresponding to
the $A_{\infty}\text{-}$structure on $\Lambda(V)$ transferred from
$\cB_W.$ Put
\begin{equation}\beta=\sum\limits_{k\geq 1}(-1)^{\frac{k(k-1)}{2}}\frac1{k!}\tilde{\Phi}^k(\alpha,\dots,\alpha).\end{equation}
Decompose $\beta$ into the sum
$\beta^0+\beta^2+\dots+\beta^{2\left[\frac n2\right]},$ with
$\beta^{2j}\in \mk[[V^{\vee}]]\otimes \Lambda^{2j}(V).$ Then $W$
can be obtained from $\beta^0$ by a formal change of variables of
type \eqref{var_change}.

\begin{remark}Note that in Theorem \ref{reconstr} we required our polynomials $W, W'$ not to have terms of order $2,$ and also required the induced
isomorphism  $H^{\cdot}(\cB_W)\to H^{\cdot}(\cB_{W'})$ to be
compatible with identifications \eqref{exterior}. The reason is
that in general Maurer-Cartan theory for DGLA's works well only in
the pro-nilpotent case. However, it should be plausible that in
the case $\mk=\C$ one can drop these assumptions. Of course, in
this case one also should drop the requirement on the change of
variables to be of type \eqref{var_change}.\end{remark}

\section{Equivalence of two LG models}
\label{LG_model}

Take $V=\C^3$ and $K\subset G\subset SL(V)$ be as before. In this section we describe two different LG models,
such that the resulting categories are equivalent.

The first one is stacky: $(V//K,W),$ where $W$ is our superpotential. The associated category $\overline{D_{sg,K}}(W^{-1}(0))$ has already been described (Corollary \ref{A'*Z_{2g+1}}).

Now we describe another LG model, which is taken in the main theorem. There is a well-known crepant resolution of the
quotient $V/K:$
\begin{equation}X=\Hilb_{K}(V)\to V/K.\end{equation}

More explicitly, $X$ is toric by \cite{CR} and is given by the following fan. Take $N\subset \R^3,$
$N=\Z^3+\Z\cdot\frac1{2g+1}(1,1,2g-1).$ Now, if we take a fan
$\Sigma$ consisting of a positive octant and its faces, then we
have $X_{\Sigma}\cong V/K$. To describe $X,$ we should subdivide
the fan $\Sigma.$ Namely, take the fan $\Sigma'$ consisting of the
cones generated by
\begin{equation}(\frac1{2g+1}(k,k,2g+1-2k),\frac1{2g+1}(k+1,k+1,2g-1-2k),(1,0,0)),\quad
0\leq k\leq g-1;\end{equation}
\begin{equation}(\frac1{2g+1}(k,k,2g+1-2k),\frac1{2g+1}(k+1,k+1,2g-1-2k),(0,1,0)),\quad
0\leq k\leq g-1;\end{equation}
\begin{equation}(\frac1{2g+1}(g,g,1),(1,0,0),(0,1,0)),\end{equation} and all their faces (see Figure 3 for the case $g=3$). Then
$X\cong X_{\Sigma'}.$

\begin{center}\begin{tabular}{c}
{\epsfbox{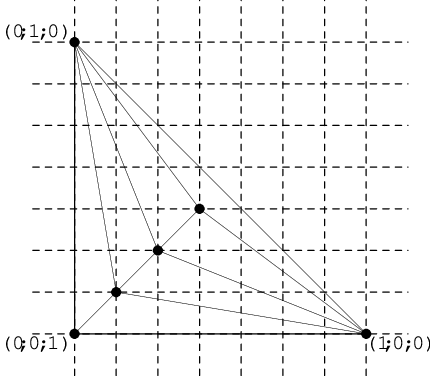}}\\\LARGE{Figure 3.}\end{tabular}\end{center}

The function $W\in\C[V^{\vee}]$ is $K\text{-}$invariant, hence
gives a function on $V/K,$ and on $X.$ The LG model $(X,W)$ is a
mirror to the genus $g$ curve. The only singular fiber of $W$ on
$X$ is $X_0=:H.$ The surface $H$ has $(g+1)$ irreducible
components $H_1,\dots,H_{g+1},$ where $H_i$ is defined below
for $1\leq i\leq g,$ and $H_{g+1}$ is the proper pre-image of
$W^{-1}(0)\subset V/K.$

The exceptional surface $H_k\subset X,$ $q\leq k\leq g,$ corresponding to the vector
$\frac1{2g+1}(k,k,2g+1-2k)\in N$ is

\begin{equation}\begin{cases}\text{the rational ruled surface }F_{2g+1-2k}\cong \PP_{\C\PP^1}(\cO\oplus\cO(2g+1-2k)) & \text{for }1\leq k\leq g-1\\
\C\PP^2 & \text{for }k=g.\end{cases}\end{equation}

The surfaces $H_i$ and $H_j$ have empty intersection if $|i-j|\geq
2.$ Further, the surfaces $H_i$ and $H_{i+1}$ intersect
transversally along the curve $C_i\subset X,$ where $1\leq i\leq
g-1.$ The curve $C_i$ is

\begin{equation}\begin{cases}\text{the "}\infty\text{-section" }
\PP_{\C\PP^1}(\cO(2g+1-2i))\subset
\PP_{\C\PP^1}(\cO\oplus\cO(2g+1-2i))\cong H_i\text{ on }H_i &
\text{for }1\leq i\leq g-1\\
\text{the "zero-section" }\PP_{\C\PP^1}(\cO)\subset
\PP_{\C\PP^1}(\cO\oplus\cO(2g-1-2i))\cong H_{i+1}\text{ on
}H_{i+1} &
\text{for }1\leq i\leq g-2\\
\text{the line on }\C\PP^2\cong H_g & \text{for }k=g.
\end{cases} \end{equation}

The divisor $H$ has simple normal crossings. We have already
described the intersections between $H_i$ for $1\leq i\leq g.$
Further, the intersection $H_i\cap H_{g+1}$ is:

\begin{equation}\begin{cases}\text{the section }\PP_{\C\PP^1}(\cO\times
(y_0y_1,y_0^{2g+1}+y_1^{2g+1}))\subset
\PP_{\C\PP^1}(\cO(2)\oplus\cO(2g+1))\cong H_1 & \text{for }i=1\\
\text{the union of two fibers
}\{y_0y_1=0\}\subset\PP_{\C\PP^1}(\cO\oplus\cO(2g+1-2i)) &
\text{for
}2\leq i\leq g-1\\
\text{a non-degenerate conic in }\C\PP^2\cong H_g & \text{for
}i=g.
\end{cases}
\end{equation}
Here $(y_0:y_1)$ are homogeneous coordinates on $\C\PP^1.$

The triple intersection $H_i\cap H_{i+1}\cap H_{g+1}$ consists of
two points for each $1\leq i\leq g-1.$ The corresponding dual CW
complex of this configuration is homeomorphic to $S^2.$

\begin{theo}\label{McKay} The triangulated category
$D_{sg}(H)$ is equivalent to $D_{sg,K}(W^{-1}(0)).$
\end{theo}

\begin{proof} This follows from \cite{BKR} and \cite{BP}, Theorem
1.1. Alternatively, Theorem is implied by \cite{QV}, Theorem 8.6.
\end{proof}

Denote by $\overline{D_{sg}}(H)$ the split-closure of the
triangulated category of singularities $D_{sg}(H).$

\begin{cor}\label{descr_of_D_sing} There is an equivalence $\overline{D_{sg}}(H)\cong \Perf(\C[K]\#\cA').$
\end{cor}

\begin{proof} This follows from Theorem \ref{McKay} and Corollary \ref{A'*Z_{2g+1}}.

\end{proof}

\section{Generalities on Fukaya categories}

This section is devoted to generalities on Fukaya $A_{\infty}$-categories of compact symplectic surfaces of genus $\geq 2.$
We follow \cite{Se}, Sections 6-10.

\subsection{The definition}

Let $M$ be a compact oriented surface of genus $g\geq 2$ with symplectic form $\omega.$ Denote by $\pi:S(TM)\to M$ the bundle of unit circles in the tangent bundle (it does not depend on the choice of Riemannian metric). Fix a $1\text{-}$form $\theta$ on $S(TM),$ such that
$d\theta=\pi^*\omega.$ In the definition of Fukaya $A_{\infty}$-category
$\cF(M),$ we need to fix the class of $\theta$ modulo exact
$1\text{-}$forms.

Consider some connected Lagrangian submanifold in $M,$ i.e. just a
connected closed curve $L\subset M.$ Denote by $\sigma:L\to S(TM)_{|L}$
the natural section, corresponding to some choice of orientation on $L.$
A curve $L$ is called balanced if $\int_L \sigma^*\theta=0.$
This property does not depend on the choice of orientation on $L.$
All contractible curves are not balanced. Further, if $L$ is not contractible,
then it is isotopic to some balanced curve $L'.$
Moreover, such $L'$ is unique up to Hamiltonian isotopy.

Fix some countable set $\cL$ of balanced curves in $M,$ such that

1) In each non-trivial isotopy class there is at least one curve from $\cL;$

2) Any two distinct curves in $\cL$ intersect transversally, and any three of them do not have common points.

The object of Fukaya $A_{\infty}$-category $\cF(M)$ are oriented balanced curves $L\in\cL,$
equipped with a {\it Spin} structure (there are only two {\it Spin} structures on a circle: trivial and non-trivial).

Now let $L_0,L_1$ be objects of $\cF(M),$ such that the underlying curves intersect transversally (i.e. are distinct).
We put
\begin{equation}\Hom_{\cF(M)}(L_0,L_1)=CF^{\cdot}(L_0,L_1)=
\bigoplus\limits_{x\in L_0\cap L_1}\C x.\end{equation}
The $\Z/2$-grading on $x\in L_0\cap L_1$ is even (resp. odd), if the local intersection number $L_0\cdot L_1$ at $x$ equals to $-1$ (resp. $1$).

Now we are going to  define the higher products in $\cF(M)$ on the transversal sequences.
Take objects $L_0,\dots,L_d$ in $\cF(M)$ (for some $d\geq 1$) with pairwise distinct underlying curves.
Choose some points $x_k\in M,$ defining basis elements in $CF^{\cdot}(L_{k-1},L_k),$ $1\leq k\leq d.$ Then we have
\begin{equation}\label{coefficients}\mu^d(x_d,\dots,x_1)=\sum\limits_{x_0\in L_0\cap L_d}m(x_0,x_1,\dots,x_d)x_0,\end{equation}
where $m(x_0,\dots,x_d)$ are integers defined in the following way.

Fix a complex structure on $M,$ which is compatible with the orientation induced by symplectic form.
Denote by $D$ the closed two-dimensional disk with standard complex structure. For given distinct points
$\zeta_0,\dots,\zeta_d\in \partial D,$ ordered anti-clockwise, denote by $\partial_i D$ the open part of the boundary between $\zeta_i$ and $\zeta_{i+1},$ where we put $\zeta_{i+d}:=\zeta_i.$ Consider holomorphic maps
$u:D\setminus\{\zeta_0,\dots,\zeta_d\}\to M$ (where $\zeta_0,\dots,\zeta_d$ depend on $u$), such that
$u(\partial_i D)\subset L_i$ for $0\leq i\leq d,$ and $u$ can be extended to a continuous map $D\to M,$ which sends $\zeta_k$ to
$x_k$ for $0\leq k\leq d.$
Further, two maps $u:D\setminus\{\zeta_0,\dots,\zeta_d\}\to M$ and $u':D\setminus\{\zeta_0',\dots,\zeta_d'\}\to M$
are called equivalent if $u=u'\circ\phi,$ where $\phi:D\to D$ is a holomorphic automorphism such that $\phi(\zeta_k')=\zeta_k.$
Each such $u$ has a virtual dimension. Denote by $\cM(x_0,\dots,x_d)$ the space of equivalence classes of the maps $u$ of virtual dimension zero.
Then each point of this moduli space is regular by \cite{Se2}, Lemma 13.2.
We define $m(x_0,\dots,x_d)$ as a sum of $\pm 1$ over all points
$u\in \cM(x_0,\dots,x_d),$ where the signs are defined as follows.

For each object $L$ of $\cF(M),$ such that {\it Spin} structure on the underlying curve is non-trivial,
we choose a point $\circ_L\in L,$ which is not the intersection point with any of the curves $\cL.$ We also fix a trivialization of this
{\it Spin} structure outside of $\circ_L.$ Note that each $u$ of virtual dimension zero is an immersion.
If the restriction of the map $u$ onto $\partial_i D$ is compatible with orientation on $L_i$ for $1\leq i\leq d,$
and the image of the boundary of $D$ does not contain any of the points $\circ_{L_i},$ then the sign with which $u$ contributes to $m(x_0,\dots,x_d),$
equals to $+1.$ Further, changing of orientation on one of the curves $L_i,$ $0< i< d,$ multiplies the sign by $(-1)^{|x_i|}.$
Changing of orientation on $L_d$ multiplies the sign by $(-1)^{|x_0|+|x_d|}.$
Also, the sign is multiplied by $(-1)^N,$
where $N$ is the number of boundary points on $D,$ which mapped to one of the points $\circ_{L_i}.$

According to \cite{Se}, the set $\cM(x_0,\dots,x_d)$ is finite, and so the definition of the coefficients $m(x_0,\dots,x_d)$
is correct.

Now consider the case when the objects $L_0,L_1\in Ob(\cF(M))$ have the same underlying curve $L.$ Fix a metric and a Morse function $f$ on $L,$ with a unique local minimum, and (hence) a unique local maximum, so that they both do not equal to the points of intersection with other curves in $\cL.$
Denote the local minimum (resp. maximum) by $e$ (resp. $q$).
We put
$$\Hom_{\cF(M)}(L_0,L_1)=CM^*(f)=\C\cdot e\oplus\C\cdot q$$
--- $\Z/2$-graded Morse space of the function $f.$ If {\it Spin} structures and orientations on $L_0$ and on $L_1$
are the same, then this is a complex with a zero differential, and the grading coincides with the Standard Morse one,
i.e. $e$ is an even morphism, and $q$ is an odd morphism. Further, if {\it Spin} structures are the same,
and orientations are different, then the parities are interchanged. Otherwise, if {\it Spin} structures are distinct, then the complex is acyclic.

Now let $L_0,\dots,L_d$ be objects of $\cF(M),$ for which any number of the underlying curves can coincide with each other.
Again, choose some basis elements $x_k\in\Hom(L_{k-1},L_k),$ $k=1,\dots,d,$ and $x_0\in\Hom(L_0,L_d).$
We want to define the integers $m(x_0,\dots,x_d)$ (which are coefficients as above) as a signed counting of points in some set $\cM(x_0,\dots,x_d).$
A point in this set is the following data.

First, this is a planar tree $T$ with $d+1$ semi-infinite edges, in which all the vertices have valency at least $3.$.
There must be fixed a bijection between the connected components of
$\R^2\setminus T$ and the set $\{L_0,\dots,L_d\},$ which is compatible with a cyclic order.
Moreover, it is required that any two connected components separated by some finite edge should correspond to objects with the same
underlying curve.

Second, for each vertex $v$ there must be given some points $\zeta_{0,v},\dots,\zeta_{|v|-1,v},$
on the boundary of $D$ (the numeration is anti-clockwise), and a holomorphic map
$u_v:D\setminus\{\zeta_{0,v},\dots,\zeta_{|v|-1,v}\}\to M,$ which maps
the boundary components to curves $L_i,$ corresponding to the connected components $\R^2\setminus T,$ whose closure contains the vertex $v.$
Again it is required that the map $u$ can be extended to a continuous map on the whole disk.

Further, for each finite edge we require the following. Suppose that it separates two areas,
corresponding to
$L_i$ and $L_j,$ where $i<j.$ Denote by $v_{\pm}$ its endpoints,
so that the pair of vectors $(v_+-v_-,W_{ij})$ is a positively oriented basis of $\R^2,$ where $W_{ij}$
is any vector which is a difference of some point in $j$-th area and some point in $i$-th area, and these points lie in different half-planes.
Further, denote by $\zeta_{i_{\pm},v_{\pm}}$ the corresponding points on the boundary of $D.$ Denote by $f_{ij}$ the
(fixed) Morse function on the corresponding Lagrangian curve. Then we require that the gradient flow of $f_{ij}$
maps (for some non-zero time) the point $u_{v_+}(\zeta_{i_+,v_+})$ to the point $u_{v_-}(\zeta_{i_-,v_-}).$

Finally, for a semi-infinite edge with endpoint $v$ and the corresponding boundary point $\zeta_{k,v}\in\partial D$ the following is required.
Denote by $x_i$ the corresponding basis element in the morphism space. If the curves corresponding to the areas separated by this edge, are distinct,
then $u_v(\zeta_{k,v})=x_i$ is the corresponding intersection point.
If they coincide and are equal to $L,$ then we require that
$$u_v(\zeta_{k,v})\in\begin{cases}W^u(x_i)\subset L, & \text{если }0<i\leq d;\\
W^s(x_0)\subset L, & \text{если }i=0.\end{cases}$$
Here for the point $x\in L$ we denote by $W^u(x)$ (resp. $W^s(x)$) the unstable (resp. stable) submanifold of $L$ with respect to the gradient flow of the Morse function.

Such a data (points $\zeta_{i,v}\in\partial D,$ maps $u_v$) has virtual dimension, and we define $\cM(x_0,\dots,x_d)$
as a set of data of virtual dimension zero.
It turns out that the moduli space is in general not regular, and in this case the definition should be modified in a suitable way (see discussion in \cite{Se}, Section 7). However,
we will need no modifications, except for the definition of the product $\mu^2$ on $\Hom_{\cF(M)}(L,L)$ for an object $L\in Ob(\cF(M)):$
$$\mu^2(e,e)=e,\quad \mu^2(q,e)=\mu^2(e,q)=-q,\quad\mu^2(q,q)=0.$$

We will define the signs in those cases in which we are interested in.
First note that the general definition becomes simpler if all the underlying curves $L_0,\dots,L_d$ are distinct except $L_{i-1}=L_i$ for some $0<i\leq d,$ or $L_0=L_d.$ In this case there is only one possible tree, and it has only one vertex and $d+1$ semi-infinite edges.
Then $m(x_0,\dots,x_d)$ is a signed count of holomorphic $d$-gons with sides on $L_j$ and with a marked point on one of the edges.
Now consider the examples which we need.

{\it Constant triangles.} Let $L_0\ne L_1.$
Then the constant triangle at any point $x\in L_0\cap L_1$
contributes to the products
\begin{equation}\mu^2(e,x),\mu^2(x,e):CF^{\cdot}(L_0,L_1)\to
CF^{\cdot}(L_0,L_1);\end{equation}
\begin{equation}\mu^2(x,x):CF^{\cdot}(L_1,L_0)\otimes CF^{\cdot}(L_0,L_1)\to
CF^{\cdot}(L_0,L_0).\end{equation}

Non-constant triangles do not contribute to these products, and taking signs into account we get
\begin{equation}\label{constsnt_1}\mu^2(x,e)=x,\quad \mu^2(e,x)=(-1)^{|x|}x,\quad
\mu^2(x,x)=(-1)^{|x|}q.\end{equation}

Analogously, we have
\begin{equation}\label{constant_2}\mu^2(e,e,)=e,\quad\mu^2(e,q)=-q,\quad\mu^2(q,e)=q,\quad \mu^2(q,q)=0.\end{equation}

{\it Non-constant triangles.}
Here we have to take {\it Spin} structures into account. Recall that for a curve $L\in\cL$ with a non-trivial {\it Spin} structure
we fix a generic point $\circ\ne
e,q,$ which does not coincide with any intersection point with any of the curves in $\cL.$

We have already considered the case when the underlying curves $L_0,\dots,L_d$ are pairwise distinct.

Another case in which we are interested in is when $L_0=L_d,$ and
the curves $L_0,\dots,L_{d-1}$ are pairwise distinct. Let $u\in
\cM(e,x_1,\dots,x_d).$ If the curves $L_1,\dots,L_d$
are oriented in accordance with the orientation of $\partial D$ (anti-clockwise), and the boundary of $u$ does not meet the points
 $\circ,$ then the corresponding sign equals to $+1.$
 Otherwise, changing of orientations and meeting with the points $\circ$ has the same affect on the sign as in the case when all the curves are distinct.

\subsection{Split-generators in Fukaya categories}

Suppose that $\cA$ is some $(\Z/2)\text{-}$graded
$A_{\infty}\text{-}$category with weak units, and $E\in
\Perf(\cA)$ is an object which split-generates $\Perf(\cA).$ Then
it is well-known that the natural $A_{\infty}\text{-}$functor
$\Hom(-,E):\Perf(\cA)\to \Perf(\End(E))$ is a quasi-equivaelence,
see \cite{Ke}.

Let $L_0,L_1$ be two objects of the Fukaya category $\cF(M),$ and
the Spin structure on $L_1$ is non-trivial. The Dehn twist
$\tau_{L_1}$ is a balanced symplectic automorphism of $M,$ hence
$\tau_{L_1}(L_0)$ is again balanced. According to \cite{Se} and
\cite{Se2}, we then have the following exact triangle in
$D^{\pi}\cF(M):$

\begin{equation}HF^{\cdot}(L_1,L_0)\otimes L_1\to L_0\to
\tau_{L_1}(L_0).\end{equation}

We will need the following two Lemmas from \cite{Se}, which we will use to prove that a given object is a generator of $D^{\pi}\cF(M).$

\begin{lemma}\label{L_i_give_L_0 }(\cite{Se}, Lemma 6.4) Let $L_1,\dots,L_r$ be objects of $\cF(M)$ whose
Spin structures are non-trivial. Suppose that $L_0$ is another
object, and $\tau_{L_r}\dots\tau_{L_1}(L_0)\cong L_0[1].$ Then
$L_0$ is split-generated by $L_1,\dots,L_r.$
\end{lemma}

\begin{lemma}\label{L_i_generate}(\cite{Se}, Lemma 6.5) Let $L_1,\dots,L_r$ be objects of $\cF(M)$ whose Spin structures are non-trivial and which are such that
$\tau_{L_r}\dots\tau_{L_1}$ is isotopic to the identity. Then they
split-generate $D^{\pi}(\cF(M)).$\end{lemma}

\subsection{Additional $\Z$-gradings}

Since $M$ is not Calabi-Yau, the $(\Z/2)\text{-}$grading on $M$
cannot be improved to $\Z\text{-}$gradings. However, it turns out that one can still put some $\Z$-grading for some Lagrangians,
and then control the $\Z$-homogeneous components of higher products.

Fix a complex structure on $M.$
Take a meromorphic section $\eta^r$ of the line bundle $\omega^{\otimes r}$
$T^*M^{\otimes r}.$ Let $D$ be its divisor. For any oriented
$L\subset M\setminus \supp(D),$ our section $\eta^r$ gives a map
\begin{equation}\label{L_to_S^1}L\to S^1,\quad x\to
\frac{\eta^r(X^{\otimes r})}{|\eta^r(X^{\otimes
r})|},\end{equation} where $X$ is a tangent vector to $L$ at $x,$
which points in the positive direction.

We define an $\frac1r\text{-}$grading on $L$ as a lift $L\to \R$ of the map \eqref{L_to_S^1}.
Let $\cF(M,D)$ be a version of Fukaya category, with the only
difference that Lagrangian submanifolds $L$ should lie in
$M\setminus \supp(D),$ and to be equipped with
$\frac1r\text{-}$grading. In particular, we have full and faithful
$A_{\infty}\text{-}$functor $\cF(M,D)\to \cF(M).$

Suppose that two objects $L_0,L_1$ of $\cF(M,D)$ have only
transversal intersection. Then each $x\in L_0\cap L_1,$ is
equipped with an integer $i^r(x).$ Namely, let $\alpha\in (0,\pi)$
be the angle counted clockwise from $TL_{0,x}$ to $TL_{1,x}.$ Let
$\alpha_0(x),$ $\alpha_1(x)$ be the values of
$\frac1r\text{-}$gradings of $L_0$ and $L_1$ at $x$ respectively. Then

\begin{equation}i^r(x):=\frac{r\alpha+\alpha_1(x)-\alpha_0(x)}{\pi}.\end{equation}

If $r$ is odd, then $i^r(x)\text{ mod }2$ coincides with the value
of $(\Z/2)\text{-}$grading on $x.$ Further, if $L_0=L_1,$ then
$i^r(e)=0,$ $i^r(q)=r.$

Let $u\in \cM(x_0,\dots,x_d)$ be a perturbed
pseudo-holomorphic polygon of virtual dimension zero, hence contributing to the higher product. For each $z\in Supp(D),$
denote by $\deg(u,z)$ the multiplicity with which $u$ hints $z.$
Then it follows from
the index formula that
\begin{equation}\label{from_index_formula}i^r(x_0)-i^r(x_1)-\dots-i^r(x_d)=r(2-d)+2\sum\limits_{z\in
Supp(D)}\ord(\eta^r,z)\deg(u,z).\end{equation}

Now suppose that for all points $z\in \supp(D)$ the order
$\ord(\eta^r,z)$ is the same positive integer $m>0.$.
With respect to our $\Z\text{-}$gradings $i^r(x),$ the higher
operations $\mu^i$ will decompose into the sum
\begin{equation}\label{decomp_of_mu^i}\mu^i=\mu^i_0+\mu^i_1+\dots,\end{equation} where $\mu^i_k,$ $k\geq 0,$
are homogeneous maps of degree $r(2-d)+2mk.$ Note that in section \ref{classification_theo} we considered precisely these conditions
on gradings, with $r=3$ and $m=2g-2.$

\subsection{Fukaya categories of orbifolds}

Suppose that finite group $\Gamma$ acts on $M$ by holomorphic (with respect to the chosen complex structure) diffeomorphisms.
Take the quotient $\bar{M}:=M/\Gamma,$ and consider it as an orbifold. Denote by $\pi:M\to\bar{M}$ the projection,
and by $\bar{D}\subset\bar{M}$ the set of orbifold points. Suppose that the $2$-form $\omega$ on $M$ and $1$-form
$\theta$ on $S(TM)$ are equivariant with respect to $\Gamma.$

Let $\bar{L}\subset\bar{M}$ be an embedded closed curve with transversal self-intersections,
such that $\pi^{-1}(L)\subset M$ is a union of $|\Gamma|$ curves, which are in general position. Denote by $L\subset M$
one of these curves, and assume that all curves $g(L),$ $g\in\Gamma,$ are contained in our countable set $\cL.$
Further, suppose that $L$ is equipped with orientation, {\it Spin} structure, Riemannian metric and a Morse function $f.$
Then we have the same data on each of the curves $g(L),$ $g\in\Gamma.$

Define an $A_{\infty}$-algebra $\End(\bar{L}).$ On the level of super-vector spaces, we put: $$\End(\bar{L}):=CM^{\cdot}(f)\oplus\bigoplus\limits_{g\in\Gamma\setminus\{1\}}CF^{\cdot}(L,g(L)).$$
For convenience, denote the summands by direct summands by $\End^{g}(L),$ where $\End^1(L)$ is the summand $CM^{\cdot}(f).$
In other words, the basis of $\End(\bar{L})$ is formed by the generators of Morse complex, and the points of self-intersections of $\bar{L}.$
Moreover, each such point gives two basis elements: even and odd.

Now, an $A_{\infty}$-structure on $\End(\bar{L})$ is defined as follows. Let $\bar{x_0},\dots,\bar{x_d}$ be basis elements of $\End(\bar{L}),$ and let $x_i\in M$ be their lifts onto $M.$ Suppose that $\bar{x_i}\in\End^{g_i}(\bar{L}).$
If $g_0\ne g_1\dots g_d,$ then the corresponding coefficient $m(\bar{x_0},\dots,\bar{x_d})$ equals to zero.
Otherwise, we put
$$m(\bar{x_0},\dots,\bar{x_d}):=m(x_0,x_1,g_1(x_2),\dots,g_1\dots g_{d-1}(x_d)).$$
Higher products are then defined by the formula \eqref{coefficients}.

Now suppose that the group $\Gamma$ is abelian, and denote by $G$ its dual group $\Hom(\Gamma,\C^*).$
We have the action of $G$ on the super-vector space $\End(\bar{L}):$ the element $h\in G$ acts on the summand
$\End^g(\bar{L})$ as multiplication by $h(g).$ This action is compatible with $A_{\infty}$-structure, because
\begin{equation}\label{homogeneous_for_Gamma}\mu^d(CF^{\cdot}(\bar{L},\bar{L})^{\gamma_d}\otimes\dots\otimes
CF^{\cdot}(\bar{L},\bar{L})^{\gamma_1})\subset
CF^{\cdot}(\bar{L},\bar{L})^{\gamma_d\dots\gamma_1}.\end{equation}

Tautologically, we have an $A_{\infty}$-isomorphism
\begin{equation}\label{orbifold_Fukaya}\C[G]\# \End(\bar{L})\cong \bigoplus\limits_{g_1,g_2\in\Gamma}\Hom_{\cF(M)}(g_1(L),g_2(L)).\end{equation}

\section{Fukaya category of a genus $g\geq 3$ curve}

It is convenient to represent the genus $g\geq 3$ curve $M$ as a
$2\text{-}$fold covering of $\C\PP^1,$ branched at $(2g+2)$
points: $(2g+1)\text{-}$th roots of unity and $0.$ Take the curves
$L_1,\dots,L_{2g+1},$ which are preimages of intervals
$[\zeta^0,\zeta^2],$
$[\zeta^1,\zeta^3],\dots,[\zeta^{2g-1},\zeta^0],$
$[\zeta^{2g},\zeta^1]$ respectively, where $\zeta=\exp(\frac{2\pi
i}{2g+1}).$ The special case $g=3$ is shown in Figure 4.

\begin{center}\begin{tabular}{c}\epsfbox{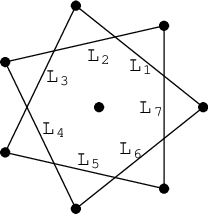}\\
\LARGE{Figure 4.}\end{tabular}\hfil\begin{tabular}{c}\epsfbox{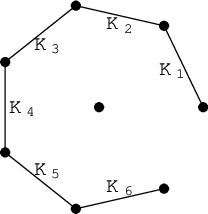}\\
\LARGE{Figure 5.}\end{tabular}\end{center}

\begin{lemma}\label{explicit_generator} The curves $L_1,\dots,L_{2g+1},$ equipped with
non-trivial Spin structures, split-generate $D^{\pi}\F(M).$
\end{lemma}

\begin{proof} Take the curves $K_1,\dots,K_{2g},$ which are
preimages of intervals $[\zeta^0,\zeta^1],
[\zeta^1,\zeta^2],\dots,[\zeta^{2g-1},\zeta^{2g}]$ respectively
(the special case $g=3$ is illustrated in Figure 5). Then by
\cite{Ma} we have $(\tau_{K_{2g}}\dots\tau_{K_1})^{4g+2}\sim \id.$
From Lemma \ref{L_i_generate}, it follows that the curves
$K_1,\dots,K_{2g},$ equipped with non-trivial spin structures,
split-generate $D^{\pi}\F(M).$ Further, it is straightforward to
check that $\tau_{L_{2g+1}}\dots\tau_{L_1}(K_1)$ is isotopic to
$K_1[1].$ Thus, it follows from Lemma \ref{L_i_give_L_0 } that
$K_1$ is split-generated by $L_1,\dots,L_{2g+1}.$ Analogously, all
the other $K_i$ are split-generated by $L_1,\dots,L_{2g+1}.$

Hence, $L_1,\dots,L_{2g+1}$ split-generate
$D^{\pi}\F(M).$\end{proof}

We now compute partially the $A_{\infty}\text{-}$algebra
$\bigoplus\limits_{1\leq i,j\leq 2g+1}CF^{\cdot}(L_i,L_j).$ Our
computation is in fact analogous to the computations in \cite{Se},
Section 10.

Take a natural $\Sigma=\Z/(2g+1)\text{-}$action on $M$ which lifts
the rotational action on $\C\PP^1.$ The quotient $M/\Sigma$ is a
sphere $\bar{M}$ with $3$ orbifold points. Denote the set of
orbifold points by $\bar{D}.$

Explicitly, the hyperelliptic curve $M$ is given (in affine chart)
by the equation

\begin{equation}y^2=z(z^{2g+1}-1).
\end{equation}
The generator of $\Sigma$ acts by the formula
\begin{equation}(y,z)\to (\zeta^{g+1}y,\zeta z).\end{equation}

We have that $\C(M)^{\Sigma}\cong \C(\frac{y}{z^{g+1}}),$ hence
$t=\frac{y}{z^{g+1}}$ is a coordinate on an affine chart
$\C\subset \C\PP^1\cong\overline{M}.$ The set $\bar{D}$ consists
of the points $t=1,$ $t=-1,$ and $t=\infty.$

Each of the curves $L_i$ projects to the same curve
$\bar{L}\subset \bar{M}.$ It lies in $\C\setminus\{-1,1\}\subset
\bar{M}$ and has the same isotopy type for all $g\geq 3.$ The case
$g=3$ is shown in Figure 6. We have natural
$A_{\infty}\text{-}$isomorphism, as in
\eqref{orbifold_Fukaya}:
\begin{equation}\bigoplus_{1\leq i,j\leq
2g+1}CF^{\cdot}(L_i,L_j)\cong \C[K]\#
CF^{\cdot}(\bar{L},\bar{L}),\end{equation}
where $K=\Hom(\Sigma,\C^*.)$ This is actually the same $K$ as in the end of section \ref{Koszul}.

The super vector space $CF^{\cdot}(\bar{L},\bar{L})$ has $8$
generators: two standard $e$ (even) and $q$ (odd), together with
three pairs ($\bar{x_i}$ (even),$x_i$ (odd)), $1\leq i\leq 3,$
coming from each self-intersection point of $\bar{L}$ (see Figure
6). Take $\widetilde{\Gamma}=\pi_1^{orb}(\bar{M}),$ and put
$\Gamma=[\widetilde{\Gamma},\widetilde{\Gamma}].$ Then $\Gamma$ is
naturally the quotient of $(\Z/(2g+1))^3$ by the diagonal subgroup
$\Z/(2g+1).$ The class of our immersed curve $\bar{L}$ in $\Gamma$
is trivial, hence the generators of $CF^{\cdot}(L,L)$ are labelled
by the weights which are elements of $\Gamma.$

Further, take a meromorphic section $\eta^3$ of
$(T^*\bar{M})^{\otimes 3},$ having double pole at each point of
$\bar{D}.$ Explicitly,
\begin{equation}\eta^3=\frac{(dt)^{\otimes 3}}{(t-1)^2(t+1)^2}.\end{equation}
Each generator of $CF^{\cdot}(\bar{L},\bar{L})$ is equipped with
additional integer grading, together with weight in $\Gamma:$

\begin{equation}
\begin{aligned}
&
\begin{array}{l|l|l|l|l}
\text{generator} & e & x_1 & x_2 & x_3 \\ \hline
\text{weight} & (0,0,0) & (1,0,0) & (0,1,0) & (0,0,1) \\
\text{index} & 0 & 1 & 1 & 1
\end{array} \\
&
\begin{array}{l|l|l|l|l}
\text{generator} & \bar{x}_1 & \bar{x}_2 & \bar{x}_3 & q \\ \hline
\text{weight} & (0,1,1) & (1,0,1) & (1,1,0) & (1,1,1) \\
& = (-1,0,0) & = (0,-1,0) & = (0,0,-1) & = (0,0,0) \\
\text{index} & 2 & 2 & 2 & 3
\end{array}
\end{aligned}
\end{equation}

\begin{center}\begin{tabular}{c}\epsfbox{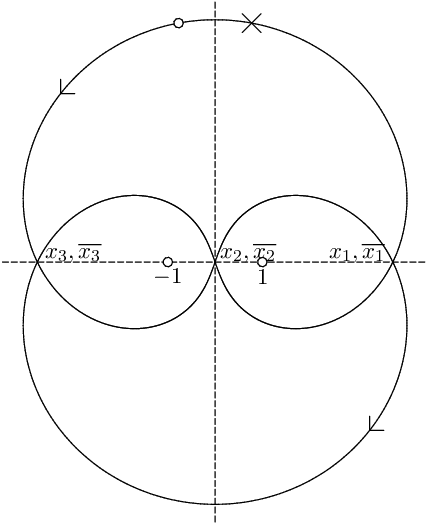}\\\LARGE{Figure 6.}\end{tabular}\end{center}

Since the $A_{\infty}\text{-}$structure is homogeneous with
respect to $\Gamma$ by \eqref{homogeneous_for_Gamma} we have that
$\mu^1=0.$

Further, the inverse image of $\eta^3$ on $M$ has three poles of
order $(2g-2).$ Therefore, according to \eqref{decomp_of_mu^i}, we
have a decomposition $\mu^i=\mu^i_0+\mu^i_1+\dots,$ where
$\mu^i_k$ has degree $6-3i+(4g-4)k.$

For degree reasons, $\mu^2_k$ vanishes for $k>0.$ Further,
according to \eqref{constsnt_1}, \eqref{constant_2}, we have
\begin{multline}\mu^2(x_i,e)=x_i=-\mu^2(e,x_i),\quad\mu^2(\bar{x_i},e)=\bar{x_i}=\mu^2(e,\bar{x_i}),\quad
\mu^2(q,e)=q=-\mu^2(e,q),\\
\mu^2(q,q)=0,\quad\mu^2(x_i,\bar{x_i})=q=-\mu^2(\bar{x_i},x_i).
\end{multline}


Further, there are only six (taking into account the ordering of
the vertices) non-constant triangles which avoid $\bar{D}.$ To
determine the sign of their contributions, choose generic points
$\circ$ on $\bar{L},$ as in Figure 6 (where $*$ denotes the point $e$). Then we have

\begin{multline}\mu^2(x_1,x_2)=\bar{x_3}=-\mu^2(x_2,x_1);\\
\mu^2(x_2,x_3)=\bar{x_1}=-\mu^2(x_3,x_2);\\
\mu^2(x_3,x_1)=\bar{x_2}=-\mu^2(x_1,x_3).\end{multline}

Further, one of the triangles (passing through $*$) can be thought
as a four-pointed disc with one of the vertex being $*.$ It gives
contribution to
\begin{equation} \mu_0^3(x_3,x_2,x_1)=-e.
\end{equation}

Further, $\mu^3_0(x_{i_1},x_{i_2},x_{i_1})=0$ for
$(i_1,i_2,i_3)\ne (3,2,1),$ since such an expression is a multiple
of $e$ (for degree reasons), and all the relevant spaces
$\bar{\cM}(e,x_{i_1},x_{i_2},x_{i_3})$ are empty.


There are six holomorphic $(2g+1)\text{-}$gons in our picture.
Namely, each point $x_i\in \bar{L}$ breaks the curve $\bar{L}$
into two loops $\gamma',\gamma''$. Choose the orientations on them
in such a way that they go anti-clockwise around the corresponding
orbifold point $t_{\gamma'}=t_{\gamma''}.$. Then for each such
loop $\gamma_j$ we have a bi-holomorphic map $v_j:S\to \bar{M},$
where $S$ is a $1\text{-}$pointed disk. The image of $v_j$ is the
area bounded by $\gamma_j$ and containing the orbifold point
$t_{\gamma_j}.$ Also require $v_j$ to map the center of $S$ to
$t_{\gamma_j}$ and the marked point to the corresponding $x_i.$
Further, define $u_j$ to be the composition of $v_j$ with the map
$z\to z^{2g+1}.$ Then $u_j$ maps the $(2g+1)\text{-}$th roots of
unity to $x_i.$

Further, each $u_j$ hits exactly one of the points of $\bar{D}$
and has $(2g+1)\text{-}$fold ramification there, and no
ramification elsewhere, which means that it lifts to a genuine
immersed $(2g+1)\text{-}$gon in $M$. We take the three
$(2g+1)\text{-}$gons that go through $*$, and determine their
contributions to $\mu^{2g+1}_1,$ namely:

\begin{equation}\mu_1^{2g+1}(x_i,\dots,x_i)=e.\end{equation}

Now identify $CF^{\cdot}(\bar{L},\bar{L})$ with $\Lambda(V),$ $V=\C^3,$ mapping $e$ to $1,$
$x_i$ to $\xi_i,$ $\bar{x_1}$ to $\xi_2\wedge \xi_3$ and
analogously for other $\bar{x_i},$ and $q$ to
$-\xi_1\wedge\xi_2\wedge\xi_3.$ Then, it follows from the above
computations and Theorem \ref{classification_theo1} that the
resulting $A_{\infty}\text{-}$structure on $\Lambda(V)$ is
$G\cong\Hom(\Gamma,\C^*)\text{-}$equivariantly
$A_{\infty}\text{-}$isomorphic to $\cA'$ from the end of the
section \ref{classification_theo}. The covering $M\to \bar{M}$ is
classified by the surjective homomorphism $\Gamma\to \Sigma,$
which is dual to the inclusion $K\subset G.$ Combining this with
Lemma \ref{explicit_generator} and \eqref{orbifold_Fukaya}, we
obtain the following

\begin{cor}\label{descr_of_DFuk} We have an equivalence $D^{\pi}\cF(M)\cong \Perf(\C[K]\#\cA').$
\end{cor}

Corollaries \ref{descr_of_D_sing} and \ref{descr_of_DFuk} imply
the main

\begin{theo} There is an equivalence $\overline{D_{sg}}(H)\cong D^{\pi}\cF(M).$
\end{theo}

\section{Appendix}

Here we prove the statement of Lemma \ref{bijection}. It is in
fact standard, but we could not find a reference. Fix some basic
field $\mk$ of characteristic zero.

In fact one can define, following Getzler \cite{Ge} a simplicial
set $\cM\cC_{\bullet}(\mg)$ for any pro-nilpotent
$L_{\infty}\text{-}$algebra, such that
$\pi_0(|\cM\cC_{\bullet}(\mg)|)$ is the set of equivalence classes
of MC solutions in $\mg.$ Further, one can prove that filtered
$L_{\infty}\text{-}$quasi-isomorphism $\Phi:\mg\to\mh$ induces a
homotopy equivalence of these simplicial sets. However, we prove
in this Appendix precisely what we need.

Let $\mg$ be a nilpotent (in the standard sense) DG Lie algebra.
Denote by $MC(\mg)$ the set of MC solutions. We have the nilpotent
group $\exp(\mg^0),$ which acts on $MC(\mg)$ as it is described in
Section \ref{MC_general}.

Now let $\mh\subset\mg$ be a DG ideal such that $[\mg,\mh]=0.$ We
have natural maps $\pi:\mg\to\mg/\mh$ and $\pi_*:MC(\mg)\to
MC(\mg/\mh).$ Then one has the following obstruction theory.

\begin{prop}\label{obstr_theory}1) There is a natural map $o_2:MC(\mg/\mh)\to H^2(\mh)$ satisfying the following property: if $\alpha\in MC(\mg/\mh),$
then the following are equivalent:

(i) The set $\pi_*^{-1}(\alpha)$ is non-empty.

(ii) $o_2(\alpha)=0.$

Moreover, if $\alpha,\beta\in MC(\mg/\mh)$ are equivalent then
$o_2(\alpha)=0$ iff $o_2(\beta)=0.$

2) Suppose that $\alpha\in MC(\mg/\mh)$ is such that the set
$(\pi_*)^{-1}(\alpha)$ is not empty. Then there is a natural
simply transitive $Z^1(\mh)$-action on the set
$(\pi_*)^{-1}(\alpha).$

3) Let $\alpha,\beta\in MC(\mg/\mh)$ and $X\in (\mg/\mh)^0$ be
such that $\exp(X)(\alpha)=\beta.$ Suppose that the set
$(\pi_*)^{-1}(\alpha)$ (and hence also $(\pi_*)^{-1}(\beta)$) is
non-empty. Take a $Z^1(\mh)\text{-}$action on
$(\pi_*)^{-1}(\beta)$ as in 2) and on $(\pi_*)^{-1}(\alpha)$
inverse to the action in 2). Then there exists a natural
$Z^1(\mh)\text{-}$equivariant map
\begin{equation}o_1^X:(\pi_*)^{-1}(\alpha)\times (\pi_*)^{-1}(\beta)\to H^1(\mh)\end{equation}
satisfying the following property: if $\tilde{\alpha}\in
(\pi_*)^{-1}(\alpha),$ $\tilde{\beta}\in (\pi_*)^{-1}(\beta)$ then
the following are equivalent:

(iii) there exists an element $\tilde{X}\in \mg^0$ such that
$\pi(\tilde{X})=X$ and
$\exp(\tilde{X})(\tilde{\alpha})=\tilde{\beta}.$

(iv) $o_1^X(\alpha,\beta)=0.$

4) Let $\alpha,\beta\in MC(\mg),$ and let $X\in (\mg/\mh)^0$ be
such that $\exp(X)(\pi_*(\alpha))=\pi_*(\beta).$ Suppose that the
set $(\pi_*)^{-1}(X)=\{\tilde{X}\in\mg^0\mid
\exp(\tilde{X})(\alpha)=\beta\}$ is non-empty. Then there is a
natural simply transitive action of $Z^0(\mh)$ on the set
$(\pi_*)^{-1}(X).$
\end{prop}
\begin{proof}1) Let $\alpha\in MC(\mg/\mh).$ Take some $\tilde{\alpha}\in \mg^1$ such that $\pi(\tilde{\alpha})=\alpha.$
Then it is easy to check that
$$\cF(\tilde{\alpha}):=\partial\tilde{\alpha}+\frac12[\tilde{\alpha},\tilde{\alpha}]\in
Z^2(\mg).$$ Define $o_2(\alpha)$ to be the class of
$\cF(\tilde{\alpha}).$

Check that this is well defined. Take some other lift
$\tilde{\alpha}'\in\mg^1$ of $\alpha.$ Since
$\alpha-\alpha'\in\mh$ is central, we have that
$\cF(\tilde{\alpha})-\cF(\tilde{\alpha}')=\partial(\tilde{\alpha}-\tilde{\alpha}').$
Therefore, $o_2(\alpha)$ is well defined. Now we prove that
(i)$\Leftrightarrow$(ii).

(i)$\Rightarrow$(ii). Let $\tilde{\alpha}\in MC(\mg)$ be such that
$\pi_*(\tilde{\alpha})=\alpha.$ Then $\cF(\tilde{\alpha})=0,$
hence $o_2(\alpha)=0.$

(ii)$\Rightarrow$(i). Let $\tilde{\alpha}\in \mg^1$ be such that
$\pi(\tilde{\alpha})=\alpha.$ Since $o_2(\alpha)=0,$ there exists
$u\in\mh^1$ such that $\cF(\tilde{\alpha})=\partial(u).$ Then
$\tilde{\alpha}-u\in MC(\mg)$ and
$\pi_*(\tilde{\alpha}-u)=\alpha.$

Now, suppose that $\alpha,\beta\in MC(\mg/\mh)$ are equivalent,
and $X\in (\mg/\mh)^0$ is such that $\exp(X)(\alpha)=\beta.$
Suppose that $o_2(\alpha)=0.$ Take some lift $\tilde{\alpha}\in
MC(\mg)$ of $\alpha,$ and a lift $\tilde{X}\in\mg^0$ of $X.$ Then
$\exp(\tilde{X})(\tilde{\alpha})\in MC(\mg)$ is a lift of $\beta,$
hence $o_2(\beta)=0.$ Analogously, vanishing of $o_2(\beta)$
implies vanishing of $o_2(\alpha).$

2) The desired action is just the translation one. It is obviously
simply transitive.

3) Let $\tilde{\alpha}\in (\pi_*)^{-1}(\alpha),$ $\tilde{\beta}\in
(\pi_*)^{-1}(\beta).$ Take some lift $\tilde{X}\in \mg^0$ of $X.$
Define $o_1^{X}(\tilde{\alpha},\tilde{\beta})$ to be the class of
$\tilde{\beta}-\exp(\tilde{X})(\alpha)$ in $H^1(\mh).$

First check that this is well defined. Let $\tilde{X}'\in\mg^0$ be
another lift of $\tilde{X}.$ Then we have that

\begin{equation}(\tilde{\beta}-\exp(\tilde{X})(\alpha))-(\tilde{\beta}-\exp(\tilde{X}')(\alpha))=
\partial(\tilde{X}-\tilde{X}').\end{equation}

Therefore, the map $o_1^X$ is well defined. It is clear that it is
$Z^1(\mh)\text{-}$equivariant. Now prove that
(iii)$\Leftrightarrow$(iv).

(iii)$\Rightarrow$(iv). It suffices to choose $\tilde{X}\in\mg^0$
such that $\exp(\tilde{X})(\tilde{\alpha})=\tilde{\beta}.$

(iv)$\Rightarrow$(iii). Choose some lift $\tilde{X}\in\mg^0$ of
$X.$ Since $o_1^X(\tilde{\alpha},\tilde{\beta})=0,$ there exists
$u\in\mh^0$ such that
$\tilde{\beta}-\exp(\tilde{X})(\alpha)=\partial u.$ Then
$\exp(\tilde{X}-u)(\alpha)=\beta$ and $\pi(\tilde{X}-u)=X.$

4) The desired action is just the translation one. It is obviously
simply transitive.
\end{proof}

Let $\mg$ be some pro-nilpotent DG Lie algebra.

Now we recall the notion of homotopy between two exponents in
$\exp(\mg^0).$ If $\alpha\in\mg^1$ is an MC solution, then we have
the deformed differential $\partial_{\alpha}(u)=\partial
u+[\alpha,u].$
\begin{defi}Let $\alpha,\alpha'\in\mg^1$ be MC solutions, and let
$X,Y\in\mg^0$ be such that
$\exp(X)\cdot\alpha=\exp(Y)\cdot\alpha=\alpha'.$ Then an element
$H\in\mg^{-1}$ is called a homotopy between $X$ and $Y$ if
\begin{equation}\exp(Y)=\exp(X)\exp(\partial_{\alpha}H).\end{equation}
\end{defi}

It is clear that for each $X$ and $u$ as in definition there
exists precisely one $Y\in \mg^0$ such that $u$ is a homotopy
between $X$ and $Y.$

Now we prove the special case of Lemma \ref{bijection}.

\begin{prop}\label{bij_special}Let $\Phi:\mg\to\mh$ be a DG filtered quasi-isomorphism of pro-nilpotent DG Lie algebras.
Then the induced map $\Phi_*:MC(\mg)/\exp(\mg^0)\to
MC(\mh)/\exp(\mh^0)$ is a bijection.\end{prop}

\begin{proof} 1) First we prove that the induced map on the equivalence classes
of MC solutions is {\it surjective.} Take some $r\geq 1.$ Denote
by $\pi_1:\mg/L_{r+1}\mg\to \mg/L_r\mg,$ $\pi_2:\mh/L_{r+1}\mh\to
\mg/L_r\mh$ the natural projections. Clearly, it suffices to prove
the following

\begin{lemma}Take some $\alpha\in MC(\mg/L_r\mg).$ Suppose that there exists $\beta\in MC(\mh/L_{r+1}\mh)$
such that $\pi_{2*}(\beta)=\Phi_*(\alpha).$ Then there exists
$\tilde{\alpha}\in MC(\mg/L_{r+1}\mg)$ and $X\in
L_r\mh^0/L_{r+1}\mh^0,$ such that
$\pi_{1*}(\tilde{\alpha})=\alpha,$ and
\begin{equation}\Phi_*(\tilde{\alpha})=\exp(X)(\beta)=\beta-\partial X.\end{equation}
\end{lemma}
\begin{proof}First, we have that $o_2(\Phi_*(\alpha))=\Phi(o_2(\alpha)).$ Since $\pi_{2*}(\beta)=\Phi_*(\alpha),$
we have by Proposition \ref{obstr_theory} that
$o_2(\Phi_*(\alpha))=0.$ Since $\Phi$ is filtered
quasi-isomorphism, we have that $o_2(\alpha)=0.$ Therefore, by
Proposition \ref{obstr_theory}, there exists some
$\tilde{\alpha}\in MC(\mg/L_{r+1}\mg),$ such that
$\pi_{1*}(\tilde{\alpha})=\alpha.$

Let $u\in Z^1(L_r\mg/L_{r+1}\mg).$ Then we have that
$o_1^0(\Phi_*(\tilde{\alpha}+u),\beta)=o_1(\Phi_*(\tilde{\alpha}),\beta)-\Phi^1(u).$
Again, since $\Phi$ is filtered quasi-isomorphism, we can choose
$u$ in such a way that $o_1^0(\Phi_*(\tilde{\alpha}+u),\beta)=0.$
In this case, by Proposition \ref{obstr_theory} 3), we have that
there exists $X\in L_r\mh^0/L_{r+1}\mh^0,$ such that
$\Phi_*(\tilde{\alpha})=\exp(X)(\beta).$ Lemma is proved.
\end{proof}

Surjectivity is proved.

2) Now, prove that our map is injective. Take some $r\geq 1.$
Denote by $\pi_1:\mg/L_{r+1}\mg\to \mg/L_r\mg,$
$\pi_2:\mh/L_{r+1}\mh\to \mg/L_r\mh$ the natural projections.
Clearly, it suffices to prove the following

\begin{lemma}\label{inj_special} Let $\alpha,\beta\in MC(\mg/L_{r+1}\mg),$ $X\in (\mg/L_r\mg)^0,$ and $Y\in MC(\mh/L_{r+1}\mh)$ be such that
$\exp(Y)(\Phi_*(\alpha))=\Phi_*(\beta),$
$\exp(X)(\pi_{1*}(\alpha))=\pi_{1*}(\beta),$ and
$\Phi(X)=\pi_2(Y).$ Then there exists some $\tilde{X}\in
(\mg/L_{r+1}\mg)^0$ such that
\begin{equation}\label{tilde_X}\pi_1(\tilde{X})=X,\quad
\exp(\tilde{X})(\alpha)=\beta,\end{equation} and $\Phi(\tilde{X})$
is homotopic to $Y$ (as a homotopy between $\Phi_*(\alpha)$ and
$\Phi_*(\beta)$).\end{lemma}
\begin{proof}First, we have that $\Phi(o_1^X(\alpha,\beta))=o_1^{\Phi(X)}(\Phi_*(\alpha),\Phi_*(\beta)).$ By
Proposition \ref{obstr_theory} 3), we have that
$o_1^{\Phi(X)}(\Phi_*(\alpha),\Phi_*(\beta))=0.$ Since $\Phi$ is
filtered quasi-isomorphism, we have that $o_1^X(\alpha,\beta)=0.$
Therefore, by Proposition \ref{obstr_theory} 3), there exists some
$\tilde{X}\in (\mg/L_{r+1}\mg)^0$ such that \eqref{tilde_X}
holds.It follows from Proposition \ref{obstr_theory} 4) and
surjectivity of the map $H^0(L_r\mg/L_{r+1}\mg)\to
H^0(L_r\mh/L_{r+1}\mh),$ that $\tilde{X}$ can be choosen in such a
way that $Y-\Phi(\tilde{X})=\partial u$ for some $u\in
(L_r\mh/L_{r+1}\mh)^{-1}.$ Then $u$ is a homotopy between
$\Phi(\tilde{X})$ and $Y.$ Lemma is proved.
\end{proof}

Injectivity is proved.
\end{proof}

To prove Lemma \ref{bijection}, we need first to modify the notion
of homotopy between MC solutions (so that it generalizes naturally
to pro-nilpotent $L_{\infty}\text{-}$algebras). Denote by
$\Omega_1$ the commutative DG algebra of polynomial differential
form on the affine line. Denote by $t$ the parameter on the line.
If $\mg$ is a DG Lie algebra, then $\mg\otimes \Omega_1$ is also a
DG Lie algebra.

In the case when $\mg$ is pro-nilpotent, we may and will consider
the completed tensor product:

\begin{equation}\mg\hat{\otimes}\Omega_1:=\lim\limits_{\leftarrow}\, (\mg/L_r\mg)\otimes\Omega_1.\end{equation}

This is also naturally filtered pro-nilpotent DG Lie algebra. We
have natural inclusion $\iota:\mg\to\mg\hat{\otimes}\Omega_1$
which is a filtered quasi-isomorphism. Further, for each $t_0\in
\mk$ we have the evaluation morphism
$\ev_{t_0}:\mg\hat{\otimes}\Omega_1\to\mg,$ which is left inverse
to $\iota,$ and hence is also filtered quasi-isomorphism.

\begin{prop}Let $\mg$ be a pro-nilpotent DG Lie algebra. Take some $\alpha,\beta\in MC(\mg).$ Then the following
are equivalent:

(i) $\alpha$ and $\beta$ are homotopic.

(ii)There exists some $A\in MC(\mg\hat{\otimes}\Omega_1)$ such
that $\ev_{0*}(A)=\alpha$ and $\ev_{1*}(A)=\beta.$\end{prop}

\begin{proof}(ii)$\Rightarrow$(i). From Proposition \ref{bij_special} we
deduce that $A$ is homotopic both to $\iota_*(\alpha)$ and
$\iota_*(\beta).$ Again by Proposition \ref{bij_special}, we have
that $\alpha$ and $\beta$ are homotopic.

(i)$\Rightarrow$(ii). Take $X\in\mg^0$ such that
$\exp(X)(\alpha)=\beta.$ Then it suffices to put
\begin{equation}A=\exp(tX)(\alpha)+X\otimes dt.\end{equation}
\end{proof}

From this moment, by a homotopy between MC solutions
$\alpha,\beta$ in the pro-nilpotent DGLA $\mg$ we mean an MC
solution $A\in \mg\hat{\otimes}\Omega_1$ such that
$\ev_{0*}(A)=\alpha$ and $\ev_{1*}(A)=\beta.$

We also modify the notion of homotopy between homotopies.

\begin{defi} Let $A,B\in MC(\mg\hat{\otimes}\Omega_1)$ be
homotopies between $\alpha,\beta\in MC(\mg).$ We call $A$ and $B$
homotopic if
\begin{equation}B=\exp(\partial_{\beta}(u)t)\exp(t(1-t)X)(A),\end{equation}
where $X\in \mg^0\hat{\otimes}\Omega_1^0,$ and $u\in \mg^{-1}.$
\end{defi}

We need to adapt the obstruction theory for our modified
homotopies.

\begin{prop}\label{obstr_modified}Let $\mg$ be a nilpotent DGLA, and $\mh\subset\mg$ be its central DG ideal, and $\pi:\mg\to\mg/\mh$ the
natural projection. Let $\alpha,\beta\in MC(\mg).$

1) There is a natural map $A\mapsto o_1^A(\alpha,\beta)$ which
assigns to each homotopy $A\in MC((\mg/\mh)\otimes \Omega_1)$
between $\pi_*(\alpha)$ and $\pi_*(\beta),$ an element
$o_1^A(\alpha,\beta)\in H^1(\mh),$ such that the following are
equivalent:

(i) There exists a homotopy $\tilde{A}\in MC(\mg\otimes\Omega_1)$
between $\alpha$ and $\beta$ such that $\pi_*(\tilde{A})=A.$

(ii) $o_1^A(\alpha,\beta)=0.$

2) Suppose that $A$ is the homotopy between $\pi_*(\alpha)$ and
$\pi_*(\beta).$ Then there is a natural transitive action of
$H^0(\mg)$ on homotopy classes of elements in the set
$(\pi_*)^{-1}(A).$ Here $(\pi_*)^{-1}(A)$ is the set of homotopies
$\tilde{A}$ between $\alpha$ and $\beta$ such that
$\pi_*(\tilde{A})=A.$
\end{prop}
\begin{proof}1) Take some element $\tilde{A}\in (\mg\otimes
\Omega_1)^1$ such that $\pi(\tilde{A})=A,$
$\ev_0(\tilde{A})=\alpha,$ $\ev_1(\tilde{A})=\beta.$ Put
$\cF(\tilde{A})=\partial \tilde{A}+\frac12[\tilde{A},\tilde{A}].$
Then $\cF(\tilde{A})$ is a cocycle in the complex $\mh\otimes
L^{\cdot}\subset \mg\otimes\Omega_1,$ where
\begin{equation}L^{\cdot}\subset \Omega_1,\quad L^0=t(1-t)\Omega_1^0,\quad L^1=\Omega_1^1.
\end{equation}

Clearly, we have that $H^0(L^{\cdot})=0,$ $H^1(L^{\cdot})=\mk,$
and the natural projection $L^1\to \mk$ is given by the formula
\begin{equation}\sum\limits_{i=0}^N a_it^idt\mapsto
\sum\limits_{i=0}^N \frac{a_i}{i+1}.
\end{equation}

We define $o_1^A(\alpha,\beta)$ to be the class of
$\cF(\tilde{A})$ in $H^1(\mh)\cong H^2(\mh\otimes L^{\cdot}).$ The
checking of correctness and equivalence (i)$\Leftrightarrow$(ii)
is analogous to that of Proposition \ref{obstr_theory} 1).

2) Suppose that the set $(\pi_*)^{-1}(A)$ is non-empty (otherwise
there is nothing to prove). It is clear from the proof of 1) that
there is a simply transitive translation action of the group
$Z^1(\mh\otimes L^{\cdot})$ on the set $(\pi_*)^{-1}(A).$ Further,
any coboundary $b$ in $\mh\otimes L^{\cdot}$ can be represented as
$\partial(\iota(\partial u)t+X),$ where $X\in \mh\otimes L^0$ and
$u\in \mh.$ Thus, we have that
\begin{equation}\tilde{A}+b=\exp(\partial_{\beta}(u)t)\exp(X)(\tilde{A})\end{equation}
--- homotopic to $\tilde{A}.$ Therefore, we have the desired
transitive action of $H^0(\mh).$
\end{proof}

Now we are able to prove Lemma \ref{bijection}.

\begin{proof}[Proof of Lemma \ref{bijection}.] Proof of
surjectivity is the same as in Proposition \ref{bij_special}.

Now we prove the injectivity. Take some $r\geq 1.$ Denote by
$\pi_1:\mg/L_{r+1}\mg\to \mg/L_r\mg,$ $\pi_2:\mh/L_{r+1}\mh\to
\mg/L_r\mh$ the natural projections. It suffices to prove the
following

\begin{lemma} Let $\alpha,\beta\in MC(\mg/L_{r+1}\mg),$ $A\in MC((\mg/L_r\mg)\otimes \Omega_1),$ and
$B\in MC((\mh/L_{r+1}\mh)\otimes\Omega_1)$ be such that $B$ is the
homotopy between $\Phi_*(\alpha)$ and $\Phi_*(\beta),$ $A$ is the
homotopy between $\pi_{1*}(\alpha)$ and $\pi_{1*}(\beta),$ and
$\Phi_*(A)=\pi_2(B).$ Then there exists some homotopy
$\tilde{A}\in MC((\mg/L_r\mg)\otimes \Omega_1)$ between $\alpha$
and $\beta$ such that $\pi_{1*}(\tilde{A})=A,$ and
$\Phi_*(\tilde{A})$ is homotopic to $B.$
\end{lemma}
\begin{proof} This follows from Proposition \ref{obstr_modified},
analogously to Lemma \ref{inj_special}.
\end{proof}
Lemma is proved
\end{proof}

\end{document}